\documentclass[10pt,reqno]{amsart}
\usepackage{graphicx,epsfig}
\usepackage{epstopdf}
\usepackage{pdfpages}
\usepackage{amssymb}
\usepackage{empheq}
\usepackage{cases}
\usepackage{amsthm,amsmath}
\usepackage{caption,lipsum}
\usepackage{stmaryrd}
\usepackage{tabularx}
\usepackage{color}
\usepackage{empheq,float}
\DeclareGraphicsExtensions{.eps}
\usepackage{amssymb,amsmath,graphicx,amsfonts,euscript,mathrsfs}
\usepackage{hyperref}
\usepackage{multirow}

\usepackage{enumerate}

\usepackage[ruled]{algorithm2e}

\DeclareSymbolFont{largesymbol}{OMX}{yhex}{m}{n}
\DeclareMathAccent{\Widehat}{\mathord}{largesymbol}{100}

\newtheorem{theorem}{Theorem}[section]
\newtheorem{lemma}[theorem]{Lemma}

\theoremstyle{definition}

\newtheorem{example}[theorem]{Example}

\theoremstyle{remark}
\newtheorem{remark}[theorem]{Remark}

\newcommand{\fe}{\mathrm{e}}

\newcommand{\bR}{{\mathbb R}}
\newcommand{\bC}{{\mathbb C}}

\newcommand{\bN}{{\mathbb N}}

\newcommand{\bx}{\mathbf{x}}

\numberwithin{equation}{section}

\begin{document}

% \title[Rotating NLS]{Computing the action ground state for the rotating nonlinear Schr\"odinger equation}
\title[Rotating nonlinear Schr\"odinger equation]{Computing the action ground state for the rotating nonlinear Schr\"odinger equation}

\author[W. Liu]{Wei Liu}
\address{\hspace*{-12pt}W.~Liu: South China Research Center for Applied Mathematics and Interdisciplinary Studies, South China Normal University, Guangzhou 510631, China}
\email{wliu@m.scnu.edu.cn}

\author[Y. Yuan]{Yongjun Yuan}
\address{\hspace*{-12pt}Y.~Yuan: MOE-LCSM, School of Mathematics and Statistics, Hunan Normal University, Changsha, Hunan 410081, China}
\email{yyj1983@hunnu.edu.cn}

\author[X. Zhao]{Xiaofei Zhao}
\address{\hspace*{-12pt}X.~Zhao: School of Mathematics and Statistics \& Computational Sciences Hubei Key Laboratory, Wuhan University, 430072 Wuhan, China}
\email{matzhxf@whu.edu.cn}

%\subjclass[2010]{Primary }
%
%\keywords{Vlasov-Poisson equation, Three dimensions, Strong magnetic field, Varying direction, Uniformly accurate method, Particle-In-Cell}

\date{}

\dedicatory{}

\begin{abstract}
We consider the computations of the action ground state for a rotating nonlinear Schr\"odinger equation. It reads as a minimization of the action functional under the Nehari constraint. In the focusing case, we identify an
equivalent formulation of the problem which simplifies the constraint. Based on it, we propose a normalized gradient flow method with asymptotic Lagrange multiplier and establish  the energy-decaying property. Popular optimization methods are also applied to gain more efficiency. In the defocusing case, we prove that the ground state can be obtained by the unconstrained minimization. Then the direct gradient flow method and unconstrained optimization methods are applied. Numerical experiments show the convergence and accuracy of the proposed methods in both cases, and comparisons on the efficiency are discussed. Finally, the relation between the action and the energy ground states are numerically investigated.
%%%% 35Q55, 65T40, 65N12, 65N35
\end{abstract}
% 35B38  Critical points of functionals in context of PDEs (e.g., energy functionals)
% 35Q55  NLS equations (nonlinear Schrödinger equations)
% 65N12  Stability and convergence of numerical methods for boundary value problems involving PDEs
% 81-08  Computational methods for problems pertaining to quantum theory

\maketitle

\section{Introduction}
The nontrivial solution of the elliptic equation of type $-\Delta \phi=g(\phi)$ with $g(0)=0$ has been studied for a long time \cite{Lions1,Lions2,Strauss}. It arises from many different physical contexts, such as the steady state of a diffusion model or the standing wave of a dispersive model. In this work, we are concerned with the following semi-linear elliptic problem in $d$ space dimensions for $d\in\bN_+$ as:
\begin{align}\label{model0}
&-\frac{1}{2}\Delta\phi(\bx)+V(\bx)\phi(\bx)+\beta |\phi(\bx)|^{p-1}\phi(\bx)-
\Omega L_z\phi(\bx)+\omega\phi(\bx)=0,\quad \phi(\bx)\not\equiv0,\quad \bx\in \bR^d,
\end{align}
where $\Omega,\,\beta,\,\omega\in\bR$ and $p>1$ are given parameters, $\bx=(x_1,\ldots,x_d)^\top$ and  $\phi(\bx):\bR^d\to\bC$ is the unknown. Moreover, $V(\bx)$ is a given function and $L_z$ is the angular momentum operator defined as
\begin{equation*}
L_z=\left\{\begin{aligned}
&i(x_2\partial_{x_1}-x_1\partial_{x_2}),\quad d\geq2,\\
&0, \qquad\qquad \qquad\quad \ \ \, d=1.
\end{aligned}\right.
\end{equation*}
When $d\geq3$, one would restrict $1<p<\frac{d+2}{d-2}$, otherwise the elliptic equation (\ref{model0}) would have no nontrivial solutions \cite{Lions1,Poh}.
With \begin{equation}\label{standing wave}
\psi(\bx,t)=\fe^{i\omega t}\phi(\bx),\quad t\geq0,\ \bx\in\bR^d,
\end{equation}
in fact (\ref{model0}) describes the stationary solution of the
rotating nonlinear Schr\"{o}dinger equation (RNLS) \cite{Baoadd0,Baoadd1} under a prescribed chemical potential $\omega\in\bR$:
\begin{equation}\label{model}
i\partial_t\psi(\bx,t)
=
-\frac{1}{2}\Delta \psi(\bx,t)
+V(\bx)\psi(\bx,t)
+\beta|\psi(\bx,t)|^{p-1}\psi(\bx,t)
-\Omega L_z \psi(\bx,t),
 \quad t>0,\ \bx\in\bR^d.
\end{equation}
In such sense, the solution $\phi(\bx)$ of (\ref{model0}) is referred as the \emph{standing wave} solution or \emph{solitary wave} solution of the RNLS. Without the rotation term, i.e., $d=1$ or $d\geq2$ with $\Omega=0$ in (\ref{model}), the RNLS widely applies in quantum mechanics, nonlinear optics and plasma physics \cite{BaoCai,Weinstein2}. With the presence of the rotation, i.e., $\Omega\neq0$ for $d=2,3$, the RNLS (\ref{model}) particularly models the Bose-Einstein Condensate in a rotational frame \cite{BaoCai,Fetter}.
Here, the function $V(\bx)$ denotes a trapping potential, e.g., a harmonic oscillator potential
$V(\bx)=\frac{1}{2}\sum_{j=1}^{d}\gamma_j^2x_j^2$ with $\gamma_j\geq0$.
The parameter $\Omega$ is interpreted as the angular velocity/rotational speed, and $\beta$ denotes the strength of nonlinear self-interaction with $\beta>0$ and $\beta<0$ representing the defocusing case and the focusing case, respectively.

Our elliptic equation (\ref{model0}) (see \cite{Lions2,Strauss} for the case without the potential and rotation terms) could have infinitely many solutions. Among all the nontrivial solutions, the one that  minimizes the \emph{action functional}
\begin{equation}\label{action}
S_{\Omega,\omega}(\phi) :=  \frac{1}{2}\|\nabla \phi\|_{L^2}^2
 +\int_{\bR^d}V|\phi|^2d\bx
 +\frac{2\beta}{p+1}
 \|\phi\|_{L^{p+1}}^{p+1}+L_\Omega(\phi)+\omega\|\phi\|_{L^2}^2
\end{equation}
with $L_\Omega(\phi):=-\Omega\int_{\bR^d}
\overline{\phi}(\bx)L_z\phi(\bx)d\bx$,
is called as the (action) \emph{ground state}, which will be denoted as $\phi_g$. Such definition of the ground state here follows \cite{Hajaiej,Lions1,Fukuizumi,Ohta,Shatah}. Apart from the ground state, the other nontrivial solutions $\phi$ of (\ref{model0}) is therefore a kind of `excited state', i.e., $S_{\Omega,\omega}(\phi)>S_{\Omega,\omega}(\phi_g)$, which is referred as the \emph{bound state} in the literature \cite{Lions1,Lions2}. Under the focusing nonlinearity ($\beta<0$), the existence of the ground state has been established in \cite{Lions1,Fukuizumi} for the non-rotating ($\Omega=0$) case of (\ref{model0}), and recently in \cite{Hajaiej} for the rotating ($\Omega\neq0$) case. For the non-rotating case, the ground state is found as a positive, smooth and exponentially localized function in space. Under certain conditions of the parameters in (\ref{model0}), e.g., the one in \cite{Lions1,Hirose}, such ground state is unique up to a phase translation which is believed to be also true in general. Therefore, in  the literature, e.g.,  \cite{Soffer1,Soffer2,Weinstein1,Weinstein2}, a ground state is also often defined as a positive and localized solution of (\ref{model0}).

The ground state/bound state solution of (\ref{model0}) has drawn a lot of research attentions. As far as we know, on the one hand, stable standing waves are useful in applications and the stability is mathematically relevant to many physical phenomena \cite{Berge,Fukuizumi,Weinstein2}. Therefore, under different parameter regimes, i.e., the range of parameters $p,d,\Omega,\omega$, many efforts have been devoted to analyzing the stability and instability of the ground state \cite{Hajaiej,Cuccagna2,Fukuizumi,Ohta,Shatah,Weinstein1,Weinstein2} and also the vortices bound state \cite{Pego}. The existing theoretical results are yet to cover all the parameter regimes, and so direct numerical simulations would be helpful. To do so, one will need to produce very accurately the standing wave solution from (\ref{model0}), and then simulate the dynamics in (\ref{model}) with (\ref{standing wave}).
On the other hand, nonlinear Schr\"odinger equations admit the multichannel dynamics  \cite{Cuccagna1,Schlag,Soffer1}, which means that asymptotically at large time, the solution of (\ref{model}) can become a linear combination of standing waves and a radiation. Such phenomenon can be governed by the modulation equations \cite{Schlag,Soffer1,Soffer2}. The modulation equations are a coupled system including (\ref{model0}), where the standing waves are exchanging `energy' $\omega=\omega(t)$ with the radiation at all times. Thus, to solve the modulation equations, e.g., the implementations in \cite{Soffer2p5,Soffer4}, an efficient and accurate algorithm for (\ref{model0}) would be vital.

While, to our best knowledge, the numerical techniques for computing the standing wave in particular the action ground state of (\ref{model0}) have not been addressed much so far. The techniques for the saddle critical points or (multiple) unstable solutions, as developed for the non-rotating case (see, e.g., \cite{MPA,HLA,LMM2001,LMM2019,SEM2005}) could be an option but is yet to cover the rotating case (\ref{model0}), and the fixed-point iteration type method \cite{Pet 1} could be another option but needs a good enough initial guess \cite{Pet 2} about the bound state of interest. If one only aims for the ground state, more effective approaches should certainly be developed and additional efforts are needed to cover the rotating case. What has been mostly addressed in the numerical literature is for a `twin' definition of the ground state \cite{Sparber,BaoCai,Lions3}: the minimizer of the energy
\begin{align}\label{energy}
 E_\Omega(u):=\frac{1}{2}\|\nabla u\|_{L^2}^2
 +\int_{\bR^d}V|u|^2d\bx+\frac{2\beta}{p+1}
 \|u\|_{L^{p+1}}^{p+1}+L_\Omega(u)
\end{align}
under a prescribed mass $\|u\|^2_{L^2}=m>0$, where the energy $E_\Omega(\cdot)$ shares the same expression as the action functional \eqref{action} after ignoring the term $\omega\|\phi\|^2_{L^2}$. We denote this minimizer by $u_g$. Note that the mass and the energy are conserved quantities in the RNLS (\ref{model}), and so this definition got preference among physicists. For computing $u_g$ and/or the exited states, many different kinds of numerical methods have been developed, including the nonlinear eigenvalue solvers \cite{Altmann,Maday,Chien,Chen,Cances,Feder}, normalized gradient flow also known as the imaginary time evolution methods \cite{BaoCai,Du,Succi,Faou,Cai,Wang,Shen0}, constrained optimization techniques \cite{ALT2017,Ostermann,Danaila1,Danaila2,Henning,Wen} and methods for saddle points \cite{CGAD,YZ2008SISC}. Among these numerical methods, the normalized gradient flow methods are particularly popular for the reason of easier applications and extensions to more complicated model setups such as the multi-components case \cite{Cai,Wang}.
 For such mass-prescribed ground state problem, the elliptic equation (\ref{model0}) is the Euler-Lagrange equation of the constrained minimization, and the chemical potential $\omega$ would be given afterwards by the ground state $u_g$ based on (\ref{model0}) as
\begin{equation}\label{omega_g}
\omega(u_g)=-\frac{1}{m}\left(\frac{1}{2}\|\nabla u_g\|_{L^2}^2
 +\int_{\bR^d}V|u_g|^2d\bx+\beta \|u_g\|_{L^{p+1}}^{p+1}
 +L_\Omega(u_g)\right).
 \end{equation}
 In contrast, the ground state problem of (\ref{action}) prescribes the chemical potential $\omega$ for $\phi_g$ and then determines the mass $m=\|\phi_g\|^2_{L^2}$.  It is not completely clear to us how the two kinds of minimization problems are related:
 \begin{equation}\label{diagram}
 \begin{array}{ccccc}\displaystyle
   \min_{\text{(\ref{model0})}}\{S_{\Omega,\omega}(\phi)\}&\longrightarrow&\phi_g&\longrightarrow&
   m\\[2pt]
   \omega&\longleftarrow &u_g&\longleftarrow&\displaystyle\min_{\|u\|^2_{L^2}=m}\{E_{\Omega}(u)\}.
 \end{array}
\end{equation}
We refer to \cite{Dovetta,Jeanjean2021} for some recent theoretical investigations in the focusing case without the rotation term.

In this paper, we aim to investigate the numerical techniques for computing the \emph{action ground state} of (\ref{action}) and explore numerically the features of the solution.
We first consider in Section \ref{sec:2} the focusing case of (\ref{action}), where we begin by reviewing the classical formulation of the problem in the literature that uses the variational characterization on a Nehari manifold \cite{Hajaiej,Fukuizumi}. By simplifying the variational characterization, we identify for the first time  an equivalent formulation of the problem, which reads as the minimization of a quadratic energy functional under an  $L^{p+1}$-spherical constraint. The new formulation leads to the simple normalized gradient flow and also the efficient numerical discretization, where we are able to obtain the unconditionally energy-decaying property. Some optimization techniques including the Barzilai-Borwein method and the conjugate gradient method are then proposed to further improve the computational efficiency, and some proper preconditioners are suggested. Then in Section \ref{sec:3}, we consider the defocusing case of (\ref{action}) which to our best knowledge has barely been addressed in the literature, and we find the story is indeed totally different from the focusing case. We begin by establishing the existence of the action ground state, and then we show that the problem in such case can be characterized by the direct minimization of the action functional without worrying about the Nehari constraint. Consequently, the direct gradient flow can be applied, and with a properly designed discretization we are able to obtain a modified action-decaying property. Corresponding preconditioned optimization methods are also given to the unconstrained minimization problem in this case. Finally, numerical results regarding the accuracy and efficiency of the proposed algorithms in the focusing and the defocusing cases are presented in Section \ref{sec:4} and some conclusions are drawn. As applications of the algorithms, the vortices patterns are captured in the ground state solution of the defocusing case, and a numerical experiment on the commutativity of the table (\ref{diagram}) is done in the end.

To present our theoretical findings, some  notations and facts are introduced below for the convenience of later use.

\

{\bf Notations and some basic facts.}
We assume $V(\mathbf{x})\geq0$ ($\forall\ \mathbf{x}\in\mathbb{R}^d$) throughout this paper, and we introduce the functional spaces
\[ L_V^2(\mathbb{R}^d) = \left\{ \phi\ :\ \int_{\mathbb{R}^d} V(\mathbf{x})|\phi(\mathbf{x})|^2 d\mathbf{x}<\infty \right\}, \quad
X=H^1(\mathbb{R}^d)\cap L_V^2(\mathbb{R}^d).  \]
Then, $X$ equipped with the inner product
\[ (u,v)_X= \int_{\mathbb{R}^d} \Big(\nabla u(\mathbf{x})\cdot\nabla\overline{v(\mathbf{x})} + \big(1+V(\mathbf{x})\big)u(\mathbf{x})\overline{v(\mathbf{x})} \Big) d\mathbf{x}, \quad\forall\ u,v\in X, \]
is a Hilbert space. It is well known that, when $V(\mathbf{x})\equiv0$ the space $X$ is exactly $H^1(\mathbb{R}^d)$ and is continuously embedded into $L^q(\mathbb{R}^d)$, where $q\in[2,\infty]$ for $d=1$, $q\in[2,\infty)$ for $d=2$, and $q\in[2,2d/(d-2)]$ for $d\geq3$. In order to obtain a compact embedding, the confining condition for $V$ is needed in our analysis {in the defocusing case}. It is stated as the following.

\begin{lemma}[\cite{BaoCai}]\label{lem:Xembed}
Assume that $V(\mathbf{x})\geq0$ ($\forall\ \mathbf{x}\in\mathbb{R}^d$) satisfies $\lim_{|\mathbf{x}|\to\infty}V(\mathbf{x}) = \infty$. Then the embedding $X\hookrightarrow L^q(\mathbb{R}^d)$ is compact, where $q\in[2,\infty]$ for $d=1$, $q\in[2,\infty)$ for $d=2$, and $q\in[2,2d/(d-2))$ for $d\geq3$.
\end{lemma}

As a direct conclusion from Young's inequality, the rotational term can be controlled as follows.
\begin{lemma}\label{lem:rotenergy}
Let $d\geq2$. For any constant $\delta>0$,
\begin{align}
\left| \Omega\int_{\mathbb{R}^d} \overline{\phi}L_z\phi\, d\mathbf{x}\right|
&\leq \int_{\mathbb{R}^d} \left(\frac{\delta}{2}|\nabla\phi|^2 +\frac{|\Omega|^2}{2\delta}(x_1^2+x_2^2)|\phi|^2 \right) d\mathbf{x}.
\end{align}
\end{lemma}

Applying the above two lemmas, we obtain the following result.
\begin{lemma}\label{lem:S-welldef}
Let $1<p<\frac{d+2}{d-2}$ for $d\geq3$ and $1<p<\infty$ for $d=1,2$. Assume that one of the following holds:
\begin{enumerate}[(i)]
  \item {$\Omega=0$ and $V(\mathbf{x})=0$ ($\forall\ \mathbf{x}\in\mathbb{R}^d$);}
  \item $\Omega=0$ and $V(\mathbf{x})\geq0$ ($\forall\ \mathbf{x}\in\mathbb{R}^d$) satisfies $\lim_{|\mathbf{x}|\to\infty}V(\mathbf{x})=\infty$;
  \item $d\geq2$, $V(\mathbf{x})=\frac12\sum_{j=1}^d\gamma_j^2x_j^2$ with $\gamma_j>0$ and $|\Omega|<\min\{\gamma_1,\gamma_2\}$.
\end{enumerate}
Then, the action functional $S_{\Omega,\omega}(\phi)$ \eqref{action} is well-defined for any $\phi\in X$.
\end{lemma}

\section{Variational characterization and numerical methods in focusing case}\label{sec:2}

In this section, we consider the focusing case of (\ref{model0}), i.e., $\beta<0$. We first review the study of the action ground state problem \eqref{action} in the literature. Then, we prove that the problem can be equivalently characterized by the minimization of a quadratic functional under an $L^{p+1}$-spherical constraint. Based on the simplified formulation, the normalized gradient flow approach and some preconditioned optimization methods are presented to compute the action ground state.

\subsection{Variational characterization with Nehari constraint}
We begin by briefly reviewing the formulation of the action ground state problem and the existence results in the literature.
The action ground state that we are interested in is the nontrivial solution of the elliptic equation (\ref{model0}) which minimizes the action functional $S_{\Omega,\omega}(\phi)$ \eqref{action}. Note that the gradient or variation of  \eqref{action} is
$$
\frac{\delta S_{\Omega,\omega}(\phi)}{\delta
\overline{\phi}}=\left(-\frac{1}{2}\Delta +V+\beta|\phi|^{p-1}-\Omega L_z
+\omega  \right)\phi=:H_\phi(\phi),
 $$
 and so (\ref{model0}) simply reads $H_\phi(\phi)=0$ by the notation.
As given in \cite{Hajaiej,Lions1,Fukuizumi}, the action ground state is then defined rigorously as
\begin{equation}\label{phi_g-def}
\phi_g\in\arg \min\{S_{\Omega,\omega}(\phi)\ : \ \phi\in X\backslash\{0\},\ H_\phi(\phi)=0\}.
\end{equation}
Since $p>1$ and $\beta<0$, for any fixed $\phi\in X\backslash\{0\}$, $\lim_{\rho\to\infty}S_{\Omega,\omega}(\rho\phi)=-\infty$. Thus the functional $S_{\Omega,\omega}$ is not bounded from below in $X$, and so it is necessary to consider $H_\phi(\phi)=0$ as a constraint. The $L^2$-inner product of the equation (\ref{model0}) with $\phi$ suggests the following \emph{Nehari functional}
\begin{equation}\label{Nehari}
 K_{\Omega,\omega}(\phi):=\frac12\|\nabla \phi\|_{L^2}^2
 +\int_{\bR^d}V|\phi|^2d\bx+\beta \|\phi\|_{L^{p+1}}^{p+1}
 +L_\Omega(\phi)+\omega\|\phi\|_{L^2}^2,
 \end{equation}
and $K_{\Omega,\omega}(\phi)=0$ defines the so-called Nehari manifold
\begin{equation}\label{NehariManifold}
\mathcal{M}:=\{\phi\in X\backslash\{0\},\,K_{\Omega,\omega}(\phi)=0\},
\end{equation}
which contains all nontrivial solutions to (\ref{model0}). By the variational argument, the action ground state defined in (\ref{phi_g-def}) can be equivalently written as the minimizer of (\ref{action}) on $\mathcal{M}$ \cite{Hajaiej,Fukuizumi}, i.e.,
 \begin{equation}\label{min}
\text{(\ref{phi_g-def})}\Longleftrightarrow\phi_g\in\arg \min\{S_{\Omega,\omega}(\phi)\ : \ \phi\in \mathcal{M}\}.
\end{equation}
To make this constrained minimization problem well-defined mathematically, we need a lower bound for $S_{\Omega,\omega}(\phi)$ and the set $\mathcal{M}$ being nonempty.

It is clear that under the constraint $K_{\Omega,\omega}(\phi)=0$, we have
\begin{equation}\label{action eq}
S_{\Omega,\omega}(\phi)=S_{\Omega,\omega}(\phi)-K_{\Omega,\omega}(\phi)= -\frac{p-1}{p+1}\beta\|\phi\|_{L^{p+1}}^{p+1}.
\end{equation}
Since $p>1$ and $\beta<0$, the action functional $S_{\Omega,\omega}$ restricted to $\mathcal{M}$ has a natural lower bound, i.e., $S_{\Omega,\omega}(\phi)\geq0$, $\forall\,\phi\in\mathcal{M}$. When the potential $V$ is further considered as the harmonic oscillator type: $V(\bx)=\frac{1}{2}\sum_{j=1}^{d}\gamma_j^2x_j^2$, the linear operator $-\frac12\Delta+V-\Omega L_z$ has the purely discrete spectrum \cite{Ueki}. If we denote
\begin{equation}\label{lambda_0}
\lambda_0=\inf\left\{\frac{1}{2}\|\nabla u\|_{L^2}^2+\int_{\bR^d}V|u|^2d\bx +L_\Omega(u): \|u\|_{L^2}=1\right\},
 \end{equation}
 then one finds that for any $\phi$,
 \begin{align}\label{Kineq}
   K_{\Omega,\omega}(\phi)\geq (\lambda_0+\omega)\|\phi\|_{L^2}^2+\beta\|\phi\|_{L^{p+1}}^{p+1}.
 \end{align}
Since $\beta<0$, one can clearly get a nontrivial solution $\phi$ for $K_{\Omega,\omega}(\phi)=0$ when $\lambda_0+\omega>0$, e.g., by a scaling $\phi=\rho u_0$ with some $\rho>0$ and $u_0$ is the minimizer of (\ref{lambda_0}). In such case, the existence of the action ground state of (\ref{min}) has been established in the non-rotating regime $d=1$ or $d\geq2$ and $\Omega=0$ in \cite{Fukuizumi}, and in the rotating regime $d\geq2$ in \cite{Hajaiej}. This is stated as follows.
\begin{lemma}[\cite{Hajaiej,Fukuizumi}]\label{lem:focusing-min-existence}
Let $V(\bx)=\frac{1}{2}\sum_{j=1}^{d}\gamma_j^2x_j^2,$ and $1<p<\frac{d+2}{d-2}$ for $d\geq3$ and $1<p<\infty$ for $d=1,2$ in \eqref{model0}. If $\beta<0$ and $|\Omega|<\min\{\gamma_1,\gamma_2\}$, then for any $\omega>-\lambda_0$, there exists a minimizer $\phi_g$ for \eqref{min}
 which solves \eqref{model0}. Moreover when $V$ is isotropic, $\phi_g$ can be chosen as a positive function after a shift in the phase.
\end{lemma}

{\begin{remark}
The result of Lemma~\ref{lem:focusing-min-existence} also holds in the non-rotating case without the potential (i.e., $\beta<0$, $\Omega=0$, and $V\equiv0$); see, e.g., \cite{Lions1,Jeanjean2003}.
\end{remark}}

To solve the minimization problem \eqref{min} on the constraint manifold $\mathcal{M}$, a first natural attempt would be a standard projected gradient flow of the form $\partial_t\phi=-H_{\phi}(\phi)+\lambda(\phi) \frac{\delta K_{\Omega,\omega}(\phi)}{\delta\overline{\phi}}$, where $\lambda(\phi)$ serves as a Lagrange multiplier to preserve the constraint $K_{\Omega,\omega}(\phi)=0$. This approach indeed works at the continuous level, but it is troublesome for numerical discretizations in general. This is mainly due to the complexity of the Nehari constraint and the Lagrange multiplier.
{Alternatively, the reader may refer to \cite{ChushanWang} for a recently proposed normalized gradient flow method for the minimization problem \eqref{min} in the non-rotating regime.}
In the next subsection, we are going to propose a simplified variational characterization for the action ground state.

\subsection{A simplified variational characterization}
Here, we consider to simplify \eqref{min} into an equivalent formulation. We denote the $L^{p+1}$ unit sphere in $X$ by $\mathcal{S}_{p+1}=\{u\in X \,:\,\|u\|_{L^{p+1}}=1\}$ and
introduce a quadratic energy functional (i.e., the quadratic part in the action functional) as
\begin{equation} \label{Q def}
% Q(u):=\int_{\mathbb{R}^d} \left(\frac12|\nabla u|^2 + \big(V+\omega\big)|u|^2 - \Omega\, \overline{u}L_zu \right) d\mathbf{x}.
Q(u):= \frac{1}{2}\|\nabla u\|_{L^2}^2+\int_{\bR^d}V|u|^2d\bx +L_\Omega(u) +\omega\|u\|_{L^2}^2.
\end{equation}
For any $\phi\in\mathcal{M}$, we have from $K_{\Omega,\omega}(\phi)=Q(\phi)+\beta\|\phi\|_{L^{p+1}}^{p+1}=0$ that $Q(\phi/\|\phi\|_{L^{p+1}})=-\beta\|\phi\|_{L^{p+1}}^{p-1}$. Further, we present the following result.

\begin{theorem}\label{thm:minQ}
Under the same assumptions as in Lemma~\ref{lem:S-welldef}, if $\beta<0$ {and $\omega>-\lambda_0$}, then the following holds:
\begin{enumerate}[(i)]
  \item the $L^{p+1}$-normalization $\phi\mapsto \phi/\|\phi\|_{L^{p+1}}$ is a bijection from $\mathcal{M}$ to $\mathcal{S}_{p+1}$;
  \item $\phi_*\in\mathcal{M}$ minimizes the action functional $S_{\Omega,\omega}$ on the Nehari manifold $\mathcal{M}$ if and only if its $L^{p+1}$-normalization $u_*:=\phi_*/\|\phi_*\|_{L^{p+1}}\in \mathcal{S}_{p+1}$ minimizes the quadratic functional $Q$ on $\mathcal{S}_{p+1}$, i.e.,
\begin{align}\label{eq:minQ}
Q(u_*)=\min_{u\in \mathcal{S}_{p+1}} Q(u).
\end{align}
\end{enumerate}
\end{theorem}
\begin{proof}
{\em (i).}
%For any $\phi\in\mathcal{M}$, we have $K_{\Omega,\omega}(\phi)=Q(\phi)+\beta\|\phi\|_{L^{p+1}}^{p+1}=0$ and therefore $Q(\phi)=-\beta\|\phi\|_{L^{p+1}}^{p+1}>0$.
If $\phi_1,\phi_2\in\mathcal{M}$ satisfy $\phi_1/\|\phi_1\|_{L^{p+1}}=\phi_2/\|\phi_2\|_{L^{p+1}}$, then
\[ \|\phi_1\|_{L^{p+1}}=\left[\frac{-1}{\beta}Q\left(\frac{\phi_1}{\|\phi_1\|_{L^{p+1}}}\right)\right]^{\frac{1}{p-1}}=\left[\frac{-1}{\beta}Q\left(\frac{\phi_2}{\|\phi_2\|_{L^{p+1}}}\right)\right]^{\frac{1}{p-1}}=\|\phi_2\|_{L^{p+1}}, \]
and $\phi_1=\phi_2$. The injectivity is verified. To prove the surjectivity, consider a $u\in \mathcal{S}_{p+1}$. It is noted that $Q(u)\geq(\omega+\lambda_0)\|u\|_{L^2}^2>0$. Defining $\phi_{u}:=\left(-Q(u)/\beta\right)^{\frac{1}{p-1}}u$, we have $\phi_u\in\mathcal{M}$ and $u=\phi_u/\|\phi_u\|_{L^{p+1}}$. The surjectivity follows from the arbitrariness of $u\in \mathcal{S}_{p+1}$.

{\em (ii).} The assertion is straightforward by applying (i) and noting that for any $\phi\in\mathcal{M}$,
% It is observed that for any $\phi\in\mathcal{M}$, $Q(\phi/\|\phi\|_{L^{p+1}})=-\beta\|\phi\|_{L^{p+1}}^{p-1}>0$ and
\[ S_{\Omega,\omega}(\phi)= -\beta\frac{p-1}{p+1}\|\phi\|_{L^{p+1}}^{p+1}=
\frac{p-1}{p+1}\left(-\beta\right)^{-\frac{2}{p-1}}\left[Q\left(\frac{\phi}{\|\phi\|_{L^{p+1}}}\right)\right]^{\frac{p+1}{p-1}}.
\qedhere\]
% Applying (i), the assertion holds.
\end{proof}

Theorem~\ref{thm:minQ} states that the constrained minimization (\ref{min}) for the action ground state $\phi_g\in\mathcal{M}$ is equivalent to the minimization of its $L^{p+1}$-normalization with respect to the quadratic functional $Q$ (\ref{Q def}). Thus in practice, one only needs to find the minimizer $u_*$ of \eqref{eq:minQ}, and then the action ground state for (\ref{min}) is obtained as
\begin{equation}\label{Q relation} \phi_g(\mathbf{x})=\left(\frac{Q(u_*)}{-\beta}\right)^{\frac{1}{p-1}}u_*(\mathbf{x}). \end{equation}
It is interesting to note in additional that the minimization problem \eqref{eq:minQ} does not involve the parameter $\beta$. Compared to (\ref{min}), now the functional and the constraint in \eqref{eq:minQ} are both simplified, which is important for numerical discretizations.

\subsection{Normalized gradient flow and its temporal discretization}
In order to solve \eqref{eq:minQ}, it is natural to consider the normalized gradient flow approach which has been widely applied for the mass-prescribed ($L^2$-normalized) ground state problem \cite{BaoCai,Du,Succi,Cai,Wang}. The continuous normalized gradient flow for \eqref{eq:minQ} reads as
\begin{align}\label{eq:pgf-minQ}
\partial_tu = \left(\frac12\Delta - V-\omega + \Omega L_z + \lambda(u)|u|^{p-1}\right) u,\quad t\geq0,
\end{align}
where $\lambda(u)$ is to preserve the constraint $\|u\|_{L^{p+1}}=1$, i.e.,
\begin{align}\label{eq:pgf-minQ-cons}
 \frac{d}{dt}\int_{\mathbb{R}^d}|u|^{p+1}d\mathbf{x}=0.
\end{align}
The constraint-preserving condition \eqref{eq:pgf-minQ-cons} implies
\begin{align}\label{eq:pgf-minQ-la}
 \lambda(u)=\frac{\int_{\mathbb{R}^d}|u|^{p-1}\overline{u}\left(-\frac12\Delta u+ (V+\omega)u - \Omega L_zu\right)d\mathbf{x}}{\int_{\mathbb{R}^d}|u|^{2p}d\mathbf{x}}.
\end{align}

Although this standard continuous normalized gradient flow \eqref{eq:pgf-minQ} after some appropriate discretization could be effective for solving \eqref{eq:minQ}, we note two clear drawbacks of it:  (i) The Lagrange multiplier $\lambda(u)$ given in \eqref{eq:pgf-minQ-la} calls for the strong regularity requirement on $u$, and so it may not be well-defined for an arbitrarily chosen initial data $u(\cdot,t=0)\in \mathcal{S}_{p+1}$; (ii) The strong nonlinearity involved in the numerator of  \eqref{eq:pgf-minQ-la} makes it difficult to construct an unconditionally energy stable linear scheme for \eqref{eq:pgf-minQ}.

It is noted that the Euler-Lagrange equation to \eqref{eq:minQ} is a nonlinear eigenvalue problem for $(\lambda,u)$ as
\begin{align}\label{eq:EL-minQ}
-\frac12\Delta u+ (V+\omega)u - \Omega L_z u = \lambda |u|^{p-1}u,\quad
\|u\|_{L^{p+1}}=1.
\end{align}
If $u_*\in \mathcal{S}_{p+1}$ is an eigenfunction, the corresponding eigenvalue $\lambda_*$ can be computed by taking the $L^2$-inner product of the first equation in \eqref{eq:EL-minQ} with $\overline{u}_*$, which yields
\begin{align}\label{eq:minQ-lamstar}
\lambda_* = \frac{Q(u_*)}{\int_{\mathbb{R}^d}|u_*|^{p+1}d\mathbf{x}}=Q(u_*).
\end{align}
Based on the observation \eqref{eq:minQ-lamstar}, we now propose a \emph{discrete normalized gradient flow with asymptotic Lagrange multiplier (GFALM)} to minimize $Q$ on $\mathcal{S}_{p+1}$.

Set $t_n=n\tau$, $n\geq0$, with $\tau>0$ a given time step. The proposed GFALM reads
\begin{align}\label{eq:gfalm-minQ}
\left\{\begin{aligned}
&\partial_tu = \left(\frac12\Delta - V-\omega + \Omega L_z + \widetilde{\lambda}(u(\cdot,t_n))|u(\cdot,t_n)|^{p-1}\right) u, \quad t\in[t_n,t_{n+1}), \\
& u(\mathbf{x},t_{n+1}):=u(\mathbf{x},t_{n+1}^+)=\frac{u(\mathbf{x},t_{n+1}^-)}{\|u(\cdot,t_{n+1}^-)\|_{L^{p+1}}},\quad n\geq0, \quad u(\mathbf{x},0)=u_0(\mathbf{x}),
\end{aligned}\right.
\end{align}
where $u_0\in \mathcal{S}_{p+1}$ is an initial guess for the minimizer of \eqref{eq:minQ} and $\widetilde{\lambda}$ is an {\em asymptotic Lagrange multiplier} defined as
\begin{align}\label{eq:gfalm-minQ-la}
\widetilde{\lambda}(u)=\frac{Q(u)}{\int_{\mathbb{R}^d}|u|^{p+1}d\mathbf{x}}.
\end{align}
Since $u(\cdot,t_n)\in \mathcal{S}_{p+1}$, we have $\widetilde{\lambda}(u(\cdot,t_n))=Q(u(\cdot,t_n))$. The asymptotic Lagrange multiplier (\ref{eq:gfalm-minQ-la}) is motivated from (\ref{eq:minQ-lamstar}). In fact in \eqref{eq:gfalm-minQ}, if we take $u(\cdot,t_n)=u_*$ which is the minimizer of \eqref{eq:minQ}, we see that $\widetilde{\lambda}(u_*)=Q(u_*)=\lambda(u_*)=\lambda_*$ is the corresponding eigenvalue as in \eqref{eq:minQ-lamstar}. This implies that the first equation in \eqref{eq:gfalm-minQ} becomes $\partial_t u\big|_{u=u_*}=\big(\frac12\Delta - V-\omega + \Omega L_z+\lambda_*|u_*|^{p-1}\big) u_*=0$ and the normalization factor in \eqref{eq:gfalm-minQ} becomes $\|u(\cdot,t_{n+1}^-)\|_{L^{p+1}}=1$. Thus, the limit equation of \eqref{eq:gfalm-minQ} when approaching the steady state asymptotically matches the Euler-Lagrange equation \eqref{eq:EL-minQ} precisely.

Thanks to the introduction of (\ref{eq:gfalm-minQ-la}) which removes the two aforementioned difficulties, the further temporal discretization for the GFALM \eqref{eq:gfalm-minQ} becomes quite flexible. For simplicity and efficiency, we adopt a backward-forward Euler scheme to discretize the GFALM \eqref{eq:gfalm-minQ} as
\begin{align}\label{eq:gfalmbf-minQ}
\left\{\begin{aligned}
&\frac{\tilde{u}^{n+1}-u^n}{\tau} = \left(\frac12\Delta -\alpha_n\right)\tilde{u}^{n+1} +\left(\alpha_n- V-\omega+ \Omega L_z + \widetilde{\lambda}(u^n)|u^n|^{p-1}\right) u^n, \\
& u^{n+1} = {\tilde{u}^{n+1}}/{\|\tilde{u}^{n+1}\|_{L^{p+1}}}, \quad n\geq0, \quad u^0=u_0\in \mathcal{S}_{p+1},
\end{aligned}\right.
\end{align}
where the parameter $\alpha_n\geq0$ serves as a stabilization factor and it can be appropriately chosen so that the time step can be selected as large as possible. We shall refer  (\ref{eq:gfalmbf-minQ}) as the GFALM-BF scheme for computing the action ground state (\ref{phi_g-def}). Its detailed implementation is outlined in Algorithm~\ref{algorithm-gfalmbf}.

\begin{algorithm}[h!] 
\caption{\bf A GFALM-BF algorithm.}\label{algorithm-gfalmbf}
 Give $u^0=u_0\in\mathcal{S}_{p+1}$, constant $\tau>0$. Set $n=0$.

\While{stopping criteria are not met}{

  Evaluate $\widetilde{\lambda}(u^n)=Q(u^n)$ and select a stabilization factor $\alpha_n$

  Solve the linear elliptic equation for $\tilde{u}^{n+1}$:
  \[ -\frac{\tau}{2}\Delta\tilde{u}^{n+1}+(1+\tau\alpha_n)\tilde{u}^{n+1} = u^n+\tau\left(\alpha_n- V-\omega+ \Omega L_z + \widetilde{\lambda}(u^n)|u^n|^{p-1}\right) u^n \]

  $u^{n+1}=\tilde{u}^{n+1}/\|\tilde{u}^{n+1}\|_{L^{p+1}}$

  $n:=n+1$
}
\end{algorithm}

For the stopping criterion, we can take either the one based on the norm of the residual
\begin{equation}\label{r_err}
r_{err}^{n,\infty}:=\left \| \left(-\frac12\Delta+V+\omega-\Omega L_z-\widetilde{\lambda}(u^n)|u^n|^{p-1}\right)u^n \right \|_\infty   \leq \varepsilon,
\end{equation}
or the energy difference
\begin{equation}\label{energy stop}
\mathcal{E}_{err}^n:= \left | Q(u^{n+1})-Q(u^n) \right| \leq \varepsilon.
\end{equation}
Our numerical experience tells that the energy based stopping criterion is easier to satisfy than the residual one.

It is worthwhile to point out that the scheme \eqref{eq:gfalmbf-minQ} is an implicit but linear scheme. At each time step, one only needs to solve a linear elliptic equation with constant coefficients (see Algorithm~\ref{algorithm-gfalmbf}), which can be done efficiently by an appropriate fast Poisson solver (e.g., the Fast Fourier Transform (FFT)). Moreover, we shall show that the scheme \eqref{eq:gfalmbf-minQ} is \textbf{unconditionally energy-decaying} when the stabilization factor $\alpha_n$ is chosen to be suitably large (stated in the theorem below). To prove the energy-decaying property and to discretize in the spatial direction, by noticing that the standing wave function $\phi(\bx)$ of the RNLS (\ref{model}) decays exponentially fast to zero when $| \bf x | \rightarrow \infty$ due to the trapping potential $V(\bf x)$, we truncate the spatial space $\mathbb{R}^d$ to a bounded domain $U\subset\mathbb{R}^d$, e.g., an interval for $d=1$ and a box domain for $d\geq2$, and impose the homogeneous Dirichlet or periodic boundary condition. In this paper, we consider the periodic boundary condition and apply the standard Fourier pseudospectral discretization \cite{Trefethen} for the spatial discretizations unless specified, where the details are omitted for brevity.
\begin{theorem}\label{thm:gfalmbf-decay}
Let $U$ be a box domain in $\bR^d$. Assume that $V\in L^{\infty}(U)$, $u^n\in H^1(U)\cap L^{\infty}(U)$ and one of the following holds:
\begin{enumerate}[(i)]
    \item $d=1$ and $\alpha_n\geq\frac12\max\left\{0,\mathrm{ess\,sup}_{\mathbf{x}\in U}\left(V(\mathbf{x}) +\omega - \widetilde{\lambda}(u^n)|u^n(\mathbf{x})|^{p-1}\right)\right\}$;
    \item $d\geq2$ and $\alpha_n\geq\frac12\max\left\{0,\mathrm{ess\,sup}_{\mathbf{x}\in U}\left(V(\mathbf{x})+\frac{|\Omega|^2}{2}(x_1^2+x_2^2) +\omega - \widetilde{\lambda}(u^n)|u^n(\mathbf{x})|^{p-1}\right)\right\}$.
\end{enumerate}
Then, the backward-forward Euler scheme \eqref{eq:gfalmbf-minQ} on the spatial domain $U$ with the homogeneous Dirichlet or periodic boundary condition has the unconditionally energy-decaying  property on \eqref{Q def}, i.e., for any $\tau>0$,
\begin{equation}\label{Q decay}
Q(u^{n+1})\leq Q(u^n).
\end{equation}
\end{theorem}
\begin{proof}
By taking the $L^2$-inner product of the first equation in \eqref{eq:gfalmbf-minQ} with $-2(\tilde{u}^{n+1}-u^n)$ and then taking the real part, we get
\begin{align*}
-2\left(\alpha_n+\frac1{\tau}\right)\|\tilde{u}^{n+1}-u^n\|_{L^2}^2
&= \frac12\|\nabla {\tilde{u}^{n+1}}\|_{L^2}^2 - \frac12\|\nabla u^n\|_{L^2}^2 + \frac12\|\nabla ({\tilde{u}^{n+1}}-u^n)\|_{L^2}^2 \\
&\quad\; +\int_{U}\left(V+\omega -\widetilde{\lambda}(u^n)|u^n|^{p-1}\right) \left(|{\tilde{u}^{n+1}}|^2 - |u^n|^2 - |{\tilde{u}^{n+1}}-u^n|^2\right) d\mathbf{x} \\
&\quad\; -\Omega\int_{U} \left(\overline{{\tilde{u}^{n+1}}}L_z{\tilde{u}^{n+1}} - \overline{u^n}L_zu^n - \overline{({\tilde{u}^{n+1}}-u^n)}L_z({\tilde{u}^{n+1}}-u^n)\right) d\mathbf{x} \\
&= Q({\tilde{u}^{n+1}})-Q(u^n) -\widetilde{\lambda}(u^n)\left(\int_{U} |u^n|^{p-1}|{\tilde{u}^{n+1}}|^2 d\mathbf{x} - \|u^n\|_{L^{p+1}}^{p+1} \right) \\
&\quad\; -\int_{U}\left(V+\omega -\widetilde{\lambda}(u^n)|u^n|^{p-1}\right) |{\tilde{u}^{n+1}}-u^n|^2 d\mathbf{x} \\
&\quad\; + \frac{1}{2}\|\nabla ({\tilde{u}^{n+1}}-u^n)\|_{L^2}^2+\Omega\int_{U} \overline{({\tilde{u}^{n+1}}-u^n)}L_z({\tilde{u}^{n+1}}-u^n)d\mathbf{x}.
\end{align*}
Note that $u^n\in \mathcal{S}_{p+1}$ and $\widetilde{\lambda}(u^n)=Q(u^n)$. Applying the assumptions on $\alpha_n$ and Lemma~\ref{lem:rotenergy} with the domain $\mathbb{R}^d$ replaced by $U$, we obtain
\[Q({\tilde{u}^{n+1}}) \leq Q(u^n)\int_{U} |u^n|^{p-1}|{\tilde{u}^{n+1}}|^2 d\mathbf{x}-\frac{1}{2}\|\nabla ({\tilde{u}^{n+1}}-u^n)\|_{L^2}^2 -\frac{2}{\tau}\|{\tilde{u}^{n+1}}-u^n\|_{L^2}^2, \]
for case (i), and
\[Q({\tilde{u}^{n+1}}) \leq Q(u^n)\int_{U} |u^n|^{p-1}|{\tilde{u}^{n+1}}|^2 d\mathbf{x} -\frac{2}{\tau}\|{\tilde{u}^{n+1}}-u^n\|_{L^2}^2, \]
for case (ii).
By H\"{o}lder's inequality, we have
\[ \int_{U} |u^n|^{p-1}|{\tilde{u}^{n+1}}|^2 d\mathbf{x}\leq \left(\int_{U} |u^n|^{p+1} d\mathbf{x}\right)^{\frac{p-1}{p+1}}\left(\int_{U} |{\tilde{u}^{n+1}}|^{p+1} d\mathbf{x}\right)^{\frac{2}{p+1}}
=\|{\tilde{u}^{n+1}}\|_{L^{p+1}}^2, \]
and therefore, for both cases (i) and (ii), $Q(u^{n+1})=Q({\tilde{u}^{n+1}})/\|{\tilde{u}^{n+1}}\|_{L^{p+1}}^2 \leq Q(u^n)$.
\end{proof}

\begin{remark}
By the similar analysis, one could also establish the energy-decaying property (\ref{Q decay}) for a semi-implicit discretization without stabilization terms:
\[ \frac{\tilde{u}^{n+1}-u^n}{\tau} = \left(\frac12\Delta - V-\omega+ \Omega L_z\right)\tilde{u}^{n+1} +\widetilde{\lambda}(u^n)|u^n|^{p-1} u^n. \]
%$u^{n+1} = \tilde{u}^{n+1}/\|\tilde{u}^{n+1}\|_{L^{p+1}}$.
However, to implement this scheme, fast solvers such as FFT cannot be directly applied due to the implicit treatment of the rotational and potential terms.
\end{remark}

The energy-decaying property (\ref{Q decay}) makes the proposed GFALM-BF scheme \eqref{eq:gfalmbf-minQ} (i.e., Algorithm~\ref{algorithm-gfalmbf}) mathematically elegant. In practice, Algorithm~\ref{algorithm-gfalmbf} can capture the action ground state very accurately which will be illustrated in Section \ref{sec:4}, while this does not stop us from considering the techniques from mathematical optimization to pursue more efficiency.

\subsection{Preconditioned optimization methods}\label{subsec opt foc}
In this subsection, we consider some popular optimization methods to solve \eqref{eq:minQ}. These methods will be shown later in Section \ref{sec:4} to gain significant computational efficiency in practice, particularly in high dimensions.

The approach is based on the iterative scheme of the following form:
\begin{equation}\label{eq:optimscheme}
 \tilde{u}^{n+1}=u^{n}+\tau_n d_n, \quad n=0,1,\ldots,
\end{equation}
followed by a projection step
\begin{equation}\label{eq:optimscheme-proj}
u^{n+1}=\tilde{u}^{n+1}/\|\tilde{u}^{n+1}\|_{L^{p+1}}.
\end{equation}
Here $d_n\in L^{p+1}$ is a descent direction and $\tau_n>0$ is a step length at the $n$-th approximate state $u^n$. A large class of optimization methods can be designed under the iterative framework \eqref{eq:optimscheme}-\eqref{eq:optimscheme-proj}. Here, we propose two kinds of efficient optimization methods with preconditions. One is the preconditioned Barzilai-Borwein (PBB) method which combines the preconditioned steepest descent (PSD) direction and the BB step length strategy \cite{BB1988IMAJNA}. The other one is the preconditioned conjugate gradient (PCG) method which adopts nonlinear CG directions with preconditioner and an optimal step length search. The practical preconditioners are suggested in the end.

In the subsequent discussions, we denote $g_n:=\big(-\frac12\Delta+V+\omega-\Omega L_z-\widetilde{\lambda}(u^n)|u^n|^{p-1}\big)u^n$ as an asymptotically approximation of the projected $L^2$-gradient (or variational derivative) of the functional $Q(u)$ (\ref{Q def}) at $u^n$. Here we use $\widetilde{\lambda}(u^n)$ from \eqref{eq:gfalm-minQ-la} instead of $\lambda(u^n)$ from \eqref{eq:pgf-minQ-la} for the reasons mentioned above.

\subsubsection{Preconditioned Barzilai-Borwein method}
Consider the PSD direction $d_n=-\mathcal{P}g_n$ in the iterative scheme \eqref{eq:optimscheme}-\eqref{eq:optimscheme-proj}, reading as
\begin{align}\label{eq:psd}
\tilde{u}^{n+1} =u^n-\tau_n\mathcal{P}g_n,\quad
u^{n+1}=\tilde{u}^{n+1}/\|\tilde{u}^{n+1}\|_{L^{p+1}}, \quad
n=0,1,\ldots,
\end{align}
where $\mathcal{P}$ is a symmetric positive-definite preconditioner which will be discussed later. Similar to the steepest descent method in Euclidean spaces, a fixed step length or monotonically decreasing step length search strategies (such as exact line search and Armijo/Goldstein/Wolfe-Powell inexact line search) could be applied for the above PSD method, but its numerical performance usually suffers from the zigzag-like iterative path and the slow convergence of the gradient descent method \cite{Yuan}.

As a special nonmonotone gradient descent method, the BB step length technique \cite{BB1988IMAJNA} is widely used to accelerate gradient-type optimization algorithms. Mimicking the BB gradient method in the optimization theory in Euclidean spaces \cite{BB1988IMAJNA}, we now propose the PBB method for the PSD iteration \eqref{eq:psd}. The idea is to treat the linear operator $\tau_n\mathcal{P}$ as an approximation of the inverse (projected) Hessian at $u^n$ and solve the quasi-Newton secant equation in the least-squares sense to get the step length. {This leads to explicitly choose $\tau_n$ as (cf. \cite{BB1988IMAJNA,Wen})
\begin{align}\label{eq:BBstep}
\tau^{\mathrm{BB}}_n:= \frac{\left|\langle y_{n-1},s_{n-1}\rangle_{L^2}\right|}{\langle y_{n-1},y_{n-1} \rangle_{L^2}}, \quad n\geq1,
\end{align}
where $s_{n-1}:=u^n-u^{n-1}$ and $y_{n-1}:=\mathcal{P}\left(g_n-g_{n-1}\right)$. Here and after, $\langle\cdot,\cdot\rangle_{L^2}$ denotes the $L^2$-inner product.
$\tau^{\mathrm{BB}}_n$ will be referred to as the BB step length.}
% This leads to explicitly choose $\tau_n$ as (cf. \cite{BB1988IMAJNA,Wen})
% \begin{align}\label{eq:BBsteps}
% \tau^{\mathrm{BB1}}_n:= \frac{\left|\langle y_{n-1},s_{n-1}\rangle_{L^2}\right|}{\langle y_{n-1},y_{n-1} \rangle_{L^2}}\quad\mbox{or}\quad
% \tau^{\mathrm{BB2}}_n:= \frac{\langle s_{n-1},s_{n-1} \rangle_{L^2}}{\left|\langle y_{n-1},s_{n-1}\rangle_{L^2}\right|}, \quad n\geq1,
% \end{align}
% where $s_{n-1}:=u^n-u^{n-1}$ and $y_{n-1}:=d_n-d_{n-1}=\mathcal{P}\left(g_n-g_{n-1}\right)$.
% $\tau^{\mathrm{BB1}}_n$ and $\tau^{\mathrm{BB2}}_n$ will be referred to as the BB step lengths.

Here, we present some practical techniques for using the BB step length. First, the BB step length in above is only defined for $n\geq1$ and an initial step $\tau_0>0$ needs to be prescribed. Second, the step length calculated by \eqref{eq:BBstep} may occasionally be too large or too small, so it needs to be truncated to a bounded interval $[\tau_{\min},\tau_{\max}]$ for some constants $0<\tau_{\min}<\tau_{\max}<\infty$. Then a practical framework of the constrained PBB algorithm is outlined in Algorithm~\ref{algorithm1} with the stopping criterion \eqref{r_err} or \eqref{energy stop}.
% the table below.

\begin{algorithm}[h!]
\caption{\bf A constrained PBB algorithm.}\label{algorithm1}
 Give $u^0=u_0\in\mathcal{S}_{p+1}$,  constants $0<\tau_{\min}<\tau_0<\tau_{\max}$. Set $n=0$.

\While{stopping criteria are not met}{

  $g_n=\left(-\frac12\Delta+V+\omega-\Omega L_z-\widetilde{\lambda}(u^n)|u^n|^{p-1}\right)u^n$

  $\tau_n=\begin{cases}
  \tau_0, &\mbox{if $n=0$ or } \langle y_{n-1},s_{n-1}\rangle_{L^2}=0, \\
  \max\{\min\{{\tau^{\mathrm{BB}}_n},\tau_{\max}\},\tau_{\min}\}, & \mbox{otherwise},
 \end{cases}$

  $\tilde{u}^{n+1}=u^n-\tau_{n}\mathcal{P} g_n$

  $u^{n+1}=\tilde{u}^{n+1}/\|\tilde{u}^{n+1}\|_{L^{p+1}}$

  $n:=n+1$
}
\end{algorithm}

% {\color{blue}For the stopping} criterion, we can take either the one based on the norm of the residual
% \begin{equation}\label{r_err}
% r_{err}^{n,\infty}:=\left \| \left(-\frac12\Delta+V+\omega-\Omega L_z-\widetilde{\lambda}(u^n)|u^n|^{p-1}\right)u^n \right \|_\infty   \leq \varepsilon,
% \end{equation}
% or the energy difference
% \begin{equation}\label{energy stop}
% \mathcal{E}_{err}^n:= \left | Q(u^{n+1})-Q(u^n) \right| \leq \varepsilon.
% \end{equation}
% Our numerical experience tells that the energy based stopping criterion converges more rapidly than the residual one.

\subsubsection{Preconditioned conjugate gradient method}
Inspired by nonlinear CG methods in optimization in Euclidean space \cite{Yuan} and the PCG method for the mass-prescribed ground state problem \cite{ALT2017}, we also consider the PCG direction in the iterative scheme \eqref{eq:optimscheme}-\eqref{eq:optimscheme-proj}:
\begin{equation}\label{pcg-direction}
d_n=\begin{cases} -\mathcal{P} g_n, & n=0,\\
-\mathcal{P} g_n+\beta_nd_{n-1},  &  n\geq 1, \end{cases}
\end{equation}
where different formulas for $\beta_n$ can be used. Typically, we set $\beta_n=\max \{\beta_n^{\mathrm{PRP}},0\}$ in \eqref{pcg-direction}, where
\begin{equation}\label{pcg-beta}
 \beta_n^{\mathrm{PRP}}=\frac{ \mathrm{Re} \big\langle g_n-g_{n-1}, \mathcal{P} g_n \big \rangle_{L^2} } { \big \langle g_{n-1}, \mathcal{P} g_{n-1} \big \rangle_{L^2} }
\end{equation}
is a generalization of the Polak-Ribi\`{e}re-Polyak formula \cite{Yuan}. To adaptively determine the optimal step length at each step, we compute $\tau_n>0$ by solving for
 \begin{equation}\label{pcg-brent}
\tau_n^{opt}=\arg \min_{\tau>0} \, Q \left(\frac{u^n+\tau d_n}{\|u^n+\tau d_n\|_{L^{p+1}}}\right).
\end{equation}
In our implementation of \eqref{pcg-brent}, the Brent's method \cite{Brent1,Brent2} which uses only the value of a target function to search the global minimum point within a given interval, is applied. Now, according to \eqref{eq:optimscheme}-\eqref{eq:optimscheme-proj}, (\ref{pcg-direction})-(\ref{pcg-beta}) and (\ref{pcg-brent}), the PCG method is summarized in Algorithm \ref{algorithm2} below with the stopping criterion taken as one of \eqref{r_err}-\eqref{energy stop}.

\begin{algorithm}[h!]
\caption{\bf A constrained PCG method.}\label{algorithm2}
Give $u^0=u_0\in\mathcal{S}_{p+1}$. Set $n=0$.

\While{stopping criteria are not met}{

  $g_n=\left(-\frac12\Delta+V+\omega-\Omega L_z-\widetilde{\lambda}(u^n)|u^n|^{p-1}\right)u^n$

  $d_n=\begin{cases} -\mathcal{P} g_n, & n=0 \\
   -\mathcal{P} g_n+\beta_nd_{n-1},  &  n\geq 1 \end{cases}$ \ with \ $\beta_n=\max \{\beta_n^{\mathrm{PRP}},0\}$

  $ \tau_n^{opt}=\arg \min_{\tau>0} Q \left(\frac{u^n+\tau d_n}{\|u^n+\tau d_n\|_{L^{p+1}}}\right)$

  $\tilde{u}^{n+1}=u^n+\tau_n^{opt} d_n$

  $u^{n+1}=\tilde{u}^{n+1}/\|\tilde{u}^{n+1}\|_{L^{p+1}}$

  $n:=n+1$
}
\end{algorithm}

\begin{remark}
We remark that, as an essentially nonmonotone method, the PBB method (Algorithm~\ref{algorithm1}) can be used in combination with certain nonmonotone convergence criterion to obtain better robustness and performance \cite{Raydan1997SIOPT}. Moreover, under the PCG framework (Algorithm~\ref{algorithm2}), other formulas of $\beta_n$ and/or some inexact search strategies for the step length could be employed to explore more efficient implementations \cite{Yuan}. In addition, the idea of the Riemannian BB method \cite{Iannazzo} and the Riemannian CG method \cite{Danaila2} could also be considered for  \eqref{eq:minQ}, which would further bring promising improvements. These subjects certainly require more systematical efforts and will be addressed in a future work.
\end{remark}

\subsubsection{Preconditioners}\label{preconditioner}
We now introduce specific preconditioners for the presented optimization schemes above to accelerate the convergence. We consider the PSD/PBB iterative scheme \eqref{eq:psd} for the presentation and the case of PCG  is similar. Actually, \eqref{eq:psd} can be reformulated compactly as
\begin{align}\label{iteration_scheme}
u^{n+1} = {B^n} u^n,
\end{align}
where ${B^n}:=\|\tilde{u}^{n+1}\|^{-1}_{L^{p+1}}\big(I-\tau_n \mathcal{P}A(u^n)\big)$ is the iteration operator (or matrix in the fully discretized level), with $I$ the identity operator and $A(u^n):=-\frac{1}{2}\Delta-\Omega L_z+(V+\omega-\widetilde{\lambda}(u^n)|u^n|^{p-1})I$. By analogy with the convergence theory of iterative algorithms for linear systems, the convergence and efficiency of \eqref{iteration_scheme} is expected to be influenced essentially by the spectral radius of {$B^n$}. Note that the normalization factor $\|\tilde{u}^{n+1}\|_{L^{p+1}}$ is an $\mathcal{O}(1)$ term for suitably small $\tau_n$, so we are mainly concerned with the Laplacian $\Delta$ and the potential $V$  in $A(u^n)$. As stated in \cite{ALT2017}, the Laplacian makes the largest eigenvalues of $A(u^n)$ behave as $\mathcal{O}(h^{-2})$ for a spatial mesh size $h$, and the harmonic potential makes the largest eigenvalues behave as $\mathcal{O}(L^{2})$ on the domain $[-L,L]^d$. Therefore, suitable preconditioners should be constructed so that the eigenvalues of the preconditioned operator $\mathcal{P}A(u^n)$ are bounded uniformly, and then the number of iterations would not depend much on the spatial resolution $h$ and the size of the domain $L$. To accomplish this task, here we use a symmetrized combined preconditioner $\mathcal{P}_{C}$, which have been applied to compute the mass-prescribed ground states of rotating Bose-Einstein condensates \cite{ALT2017}. $\mathcal{P}_{C}$ reads as
\begin{equation}\label{P-V-ab}
\mathcal{P}_{C}=\mathcal{P}_{V}^{1/2} \mathcal{P}_{\Delta} \mathcal{P}_{V}^{1/2},
\end{equation}
where $\mathcal{P}_{\Delta}=\left(\alpha_{\Delta}-\frac12\Delta\right)^{-1}$ and $\mathcal{P}_{V}=\big(\alpha_{V}+V\big)^{-1}$, with $\alpha_{\Delta}$ and $\alpha_{V}$ two positive shifting parameters. We numerically found it efficient to take
\begin{align*}
\alpha_{\Delta}=\alpha_{V}=\int \left( \frac{1}{2}|\nabla u_n|^2+\left(V+|\omega| \right) |u_n|^2\right) d\bf x.
\end{align*}
Note that the result of the operator $\mathcal{P}_{V}^{1/2}$ acting on a function $f$ at any point $\bx$ is given by $\mathcal{P}_{V}^{1/2} f(\bx)=f(\bx)/\sqrt{\alpha_{V}+V(\bx)}$, and the action of $\mathcal{P}_{\Delta}$ on a function is to solve a linear equation with constant coefficient which can be done very efficiently by using, e.g., the FFT. We refer to \cite{ALT2017} for more details.

\section{Variational characterization and numerical methods in defocusing case}\label{sec:3}
In this section, we study the defocusing case of (\ref{action}), i.e., $\beta>0$, where the results are very different from the focusing case discussed in Section~\ref{sec:2}. We first prove the existence of a global minimizer for the action functional $S_{\Omega,\omega}$ \eqref{action}  in $X=H^1(\mathbb{R}^d)\cap L_V^2(\mathbb{R}^d)$ under some weak assumptions on the potential $V$ and the rotational speed $\Omega$. Then we show that the action ground state in the defocusing case can be characterized by the direct minimization of $S_{\Omega,\omega}$ in $X$. Based on this theoretical result, the direct gradient flow approach is adopted and analyzed to numerically compute the action ground state. The preconditioned optimization methods are presented in the end as well.

\subsection{Existence and variational characterization via unconstrained minimization}
As far as we know, the action ground state of (\ref{action}) in the defocusing case has not been widely studied in the literature. Here we mention that \cite{Soffer1,Soffer2} proved the existence of a positive standing wave solution in the non-rotating case of (\ref{model0}). Thus, we begin by investigating the existence of the action ground state of (\ref{action}).  The following result states that the global minimizer of \eqref{action} exists and the Nehari constraint can be removed.

\begin{theorem}\label{thm:exist}
Let $\beta>0$, $\omega<-\lambda_0$, $1<p<(d+2)/(d-2)$ for $d\geq3$ and $1<p<\infty$ for $d=1,2$ in \eqref{action}, and let one of the following hold:
\begin{enumerate}[(i)]
  \item $\Omega=0$, $V(\mathbf{x})\geq 0\; (\forall\,\mathbf{x}\in\mathbb{R}^d)$, $\lim_{|\mathbf{x}|\to\infty}V(\mathbf{x}) = \infty$;
  \item $0<|\Omega|<\min\{\gamma_1,\gamma_2\}$, $V(\mathbf{x})=\frac{1}{2}\sum_{j=1}^d \gamma_j^2x_j^2$, $d\geq2$.
\end{enumerate}
Then, there exists a $\phi_g\in X$ such that
\begin{align}\label{defocuse thm}
S_{\Omega,\omega}(\phi_g)=\inf_{\phi\in X} S_{\Omega,\omega}(\phi)=\inf_{\phi\in\mathcal{M}} S_{\Omega,\omega}(\phi),
\end{align}
with $\mathcal{M}$ the Nehari manifold \eqref{NehariManifold}.
\end{theorem}

The proof of Theorem~\ref{thm:exist} will be done with the help of the following lemmas. We first establish the lower bound of the action functional $S_{\Omega,\omega}$.
\begin{lemma}\label{lem:LBS}
Under assumptions in Theorem~\ref{thm:exist}, the action functional $S_{\Omega,\omega}$ \eqref{action} is bounded from below, i.e., $\inf_{\phi\in X} S_{\Omega,\omega}(\phi)>-\infty$.
\end{lemma}
\begin{proof}
Let us start with case (i). Since $\lim_{|\mathbf{x}|\to\infty}V(\mathbf{x}) = \infty$, there exists a sufficiently large $R>0$ such that $V(\mathbf{x})+\omega>0$ when $|\mathbf{x}|>R$. It follows from $\beta>0$ and $V(\mathbf{x})\geq0$ that
\begin{align*}
S_{\Omega,\omega}(\phi)
&\geq \omega\int_{|\mathbf{x}|\leq R} |\phi|^2 d\mathbf{x} + \frac{2\beta}{p+1} \int_{|\mathbf{x}|\leq R}|\phi|^{p+1} d\mathbf{x}.
\end{align*}
By H\"{o}lder's inequality, there exists a constant $C>0$ depending only on $d,p,R$ such that
\begin{align*}
S_{\Omega,\omega}(\phi)
&\geq \omega\int_{|\mathbf{x}|\leq R} |\phi|^2 d\mathbf{x} + \beta C \left(\int_{|\mathbf{x}|\leq R} |\phi|^2 d\mathbf{x}\right)^{\frac{p+1}{2}}.
\end{align*}
Therefore, $$\inf_{\phi\in X}S_{\Omega,\omega}(\phi) \geq \min_{t\geq0} \big\{\omega t + \beta C t^{\frac{p+1}{2}} \big\} > -\infty,$$ which is the assertion for case (i).

For case (ii), we have from Lemma~\ref{lem:rotenergy} that $S_{\Omega,\omega}(\phi)\geq \int_{\mathbb{R}^d} \big((V_{\Omega}+\omega)|\phi|^2 + \frac{2\beta}{p+1}|\phi|^{p+1} \big) d\mathbf{x}$, where $V_{\Omega}(\mathbf{x}):=\frac12\sum_{j=1}^d\gamma_j^2x_j^2-\frac{|\Omega|^2}{2}(x_1^2+x_2^2)$. It is observed that $V_{\Omega}(\mathbf{x})\geq0$ and $\lim_{|\mathbf{x}|\to\infty}V_{\Omega}(\mathbf{x}) = \infty$ since $|\Omega|<\min\{\gamma_1,\gamma_2\}$. The proof is completed by utilizing the same argument for case (i).
\end{proof}

Define the sublevel set
\begin{align}
\mathcal{S}_{\leq0} := \left\{ \phi\in X: S_{\Omega,\omega}(\phi)\leq0 \right\},
\end{align}
and we have the following result.
\begin{lemma}\label{lem:levelset}
Under assumptions in Theorem~\ref{thm:exist}, the following things hold:
\begin{enumerate}[{\rm(a)}]
  \item $\mathcal{M}$ is nonempty;
  \item $\mathcal{M}\subset\mathcal{S}_{\leq0}$;
  \item $\inf_{\phi\in\mathcal{M}} S_{\Omega,\omega}(\phi)<0$;
  \item $\mathcal{S}_{\leq0}$ is uniformly bounded in $X$.
\end{enumerate}
\end{lemma}
\begin{proof}
Since $\omega<-\lambda_0$, the definition of $\lambda_0$ \eqref{lambda_0} implies that for $\varepsilon_0:=-(\lambda_0+\omega)/2>0$, there exists $u_0$ with $\|u_0\|_{L^2}=1$ such that $\int_{\mathbb{R}^d}\left(\frac12|\nabla u_0|^2+V|u_0|^2-\Omega\overline{u}_0L_zu_0\right)d\mathbf{x}<\lambda_0+\varepsilon_0$. Then, for $\forall\rho>0$,
\[ K_{\Omega,\omega}(\rho u_0)<(\lambda_0+\varepsilon_0+\omega)\rho^2+\beta\rho^{p+1}\|u_0\|_{L^{p+1}}^{p+1}=-\varepsilon_0\rho^2+\beta\rho^{p+1}\|u_0\|_{L^{p+1}}^{p+1}. \]
Clearly, $K_{\Omega,\omega}(\rho u_0)<-\varepsilon_0\rho^2/2<0$ for all sufficiently small $\rho>0$. On the other hand, we have
\[ K_{\Omega,\omega}(\rho u_0)\geq (\lambda_0+\omega)\rho^2+\beta\rho^{p+1}\|u_0\|_{L^{p+1}}^{p+1}\to+\infty,\quad \mbox{as}\quad \rho\to+\infty. \]
By the continuity of $K_{\Omega,\omega}(\rho u_0)$ with respect to $\rho$, there exists a $\rho_*>0$ such that $K_{\Omega,\omega}(\rho_*u_0)=0$ and $\rho_*u_0\in\mathcal{M}$. As a result, $\mathcal{M}$ is nonempty and (a) is obtained.
Moreover, for any $\phi\in\mathcal{M}$, we have $S_{\Omega,\omega}(\phi)=S_{\Omega,\omega}(\phi)-K_{\Omega,\omega}(\phi) =-\frac{p-1}{p+1}\beta\|\phi\|_{L^{p+1}}^{p+1} < 0$, which leads to (b) and (c).

The rest is to verify (d).
Firstly, we consider case (i). Note that $\Omega=0$, $V(\mathbf{x})\geq 0$ ($\forall\,\mathbf{x}\in\mathbb{R}^d$) and {$V(\mathbf{x})+\omega\geq1$} ($\forall\,|\mathbf{x}|>R_1$) for a suitably large constant $R_1>0$. For any $\phi\in\mathcal{S}_{\leq0}$, we have
\begin{align*}
0\geq S_{\Omega,\omega}(\phi)
&\geq \int_{\mathbb{R}^d} \left(\big(V+\omega\big)|\phi|^2 +\frac{2\beta}{p+1} |\phi|^{p+1} \right) d\mathbf{x} \\
&\geq \int_{|\mathbf{x}|>R_1} |\phi|^2 d\mathbf{x} + \int_{|\mathbf{x}|\leq R_1} \left(\omega|\phi|^2 +\frac{2\beta}{p+1} |\phi|^{p+1}\right) d\mathbf{x}.
\end{align*}
By H\"{o}lder's inequality, there exists a constant $C_1>0$ depending only on $d,p,R_1$ such that
\begin{align*}
0\geq \int_{|\mathbf{x}|\leq R_1} \left(\omega|\phi|^2 +\frac{2\beta}{p+1} |\phi|^{p+1}\right) d\mathbf{x}
&\geq \omega\int_{|\mathbf{x}|\leq R_1} |\phi|^2 d\mathbf{x} + \beta C_1 \left(\int_{|\mathbf{x}|\leq R_1} |\phi|^2 d\mathbf{x}\right)^{\frac{p+1}{2}} \geq s_0,
\end{align*}
with $s_0:=\min_{t\geq0} \big\{ \omega t + \beta C_1 t^{\frac{p+1}{2}} \big\}>-\infty$. This implies that
\[ \int_{|\mathbf{x}|\leq R_1} |\phi|^2 d\mathbf{x} \leq \left(\frac{-\omega}{\beta\,C_1}\right)^{\frac{2}{p-1}}, \]
and
\[
\int_{|\mathbf{x}|>R_1} |\phi|^2 d\mathbf{x}
\leq S_{\Omega,\omega}(\phi)-\int_{|\mathbf{x}|\leq R_1} \left(\omega|\phi|^2 +\frac{2\beta}{p+1} |\phi|^{p+1}\right) d\mathbf{x}
\leq -s_0.
\]
Therefore, $\|\phi\|_{L^2}$ is uniformly bounded. Furthermore, from $S_{\Omega,\omega}(\phi)\leq0$ with $\Omega=0$ and $\beta>0$, we can establish the uniform bound for $\|\nabla\phi\|_{L^2}$ and $\|\phi\|_{L_V^2}$. Therefore, we have the uniform bound for $\|\phi\|_X$. The arbitrariness of $\phi\in\mathcal{S}_{\leq0}$ yields that $\mathcal{S}_{\leq0}$ is uniformly bounded in $X$ for case (i).

Let us now consider case (ii). Noting that $|\Omega|<\min\{\gamma_1,\gamma_2\}$, applying Lemma~\ref{lem:rotenergy} with a constant $\delta$ satisfying $\big(\frac{|\Omega|}{\min\{\gamma_1,\gamma_2\}}\big)^2<\delta<1$, we see that every $\phi\in\mathcal{S}_{\leq0}$ satisfies
\begin{align*}
0\geq S_{\Omega,\omega}(\phi)
&\geq \int_{\mathbb{R}^d} \left(\frac{1-\delta}{2}|\nabla\phi|^2 + \big(V_{\Omega,\delta}+\omega\big)|\phi|^2 + \frac{2\beta}{p+1}|\phi|^{p+1} \right) d\mathbf{x},
\end{align*}
where $V_{\Omega,\delta}(\mathbf{x})=\frac12\sum_{j=1}^d\gamma_j^2x_j^2-\frac{|\Omega|^2(x_1^2+x_2^2)}{2\delta}$. It is observed that $V_{\Omega,\delta}(\mathbf{x})\geq0$ and $\lim_{|\mathbf{x}|\to\infty}V_{\Omega,\delta}(\mathbf{x})=\infty$, and then the assertion can be proved by the same manner for case (i).
\end{proof}

With the above preparations, we now apply the minimizing sequence method to prove Theorem~\ref{thm:exist}.

\begin{proof}[Proof of Theorem~\ref{thm:exist}]
According to Lemma~\ref{lem:LBS}, we have $c:=\inf_{\phi\in X} S_{\Omega,\omega}(\phi) >-\infty$. By Lemma~\ref{lem:levelset}, $\mathcal{M}$ is nonempty and $c\leq \inf_{\phi\in\mathcal{M}} S_{\Omega,\omega}(\phi)<0$. Noting that $S_{\Omega,\omega}(0)=0>c$ and $\mathcal{M}$ contains all nontrivial critical points of $S_{\Omega,\omega}$, so if the infimum $c=\inf_{\phi\in X} S_{\Omega,\omega}(\phi)$ is attained at some $\phi_g\in X$, then $\phi_g\in\mathcal{M}$ and $S_{\Omega,\omega}(\phi_g)=\inf_{\phi\in\mathcal{M}} S_{\Omega,\omega}(\phi)$. Hence, we need only to verify the existence of an unconstrained global minimizer $\phi_g$ such that $S_{\Omega,\omega}(\phi_g)=c$.

By Lemma~\ref{lem:LBS}-\ref{lem:levelset}, we can take a sequence $\{\phi^n\}\subset\mathcal{S}_{\leq0}$ minimizing $S_{\Omega,\omega}$, i.e.,
\begin{align}\label{eq:exist-slimit}
 \lim_{n\to\infty} S_{\Omega,\omega}(\phi^n) = c,
\end{align}
and it is uniformly bounded in $X=H^1(\mathbb{R}^d)\cap L_V^2(\mathbb{R}^d)$. Then there exists a subsequence (still denoted as $\{\phi^n\}$ for simplicity) in $X$ weakly converging to some $\phi^{\infty}\in X$. Lemma~\ref{lem:Xembed} leads to
\begin{align}\label{eq:exist-ulimitLp}
\phi^n \to \phi^{\infty}\quad \mbox{strongly in}\;\; L^2(\mathbb{R}^d)\cap L^{p+1}(\mathbb{R}^d).
\end{align}
On the other hand, by the weak lower-semicontinuity of the $H^1$ and $L_V^2$ norms, we have
\begin{align}\label{eq:exist-wlscX}
\liminf_{n\to\infty} \|\phi^n\|^2_{H^1} \geq \|\phi^{\infty}\|^2_{H^1} \quad\mbox{and}\quad
\liminf_{n\to\infty} \|\phi^n\|^2_{L_V^2} \geq \|\phi^{\infty}\|^2_{L_V^2}.
\end{align}
Combining \eqref{eq:exist-slimit}-\eqref{eq:exist-wlscX}, we obtain
\begin{align*}
c=\lim_{n\to\infty} S_{\Omega,\omega}(\phi^n) =\liminf_{n\to\infty} S_{\Omega,\omega}(\phi^n) \geq S_{\Omega,\omega}(\phi^{\infty}),
\end{align*}
which means that $S_{\Omega,\omega}(\phi^{\infty})=c$ and $\phi^{\infty}$ is an unconstrained global minimizer of $S_{\Omega,\omega}$ in $X$.
\end{proof}

\subsection{Gradient flow and its temporal discretization}
Theorem~\ref{thm:exist} indeed states that the action ground state in the defocusing case can be obtained by minimizing the action functional $S_{\Omega,\omega}$ in the whole space of $X$, without using the Nehari manifold (\ref{NehariManifold}). Thus, the problem is simplified to an unconstrained minimization of (\ref{action}),  which can be done by a direct gradient flow:
\begin{align}\label{eq:defocusing-gf}
\partial_t\phi=-\frac{\delta S_{\Omega,\omega}(\phi)}{\delta\overline{\phi}} = \frac12\Delta\phi-V\phi-\beta|\phi|^{p-1}\phi +\Omega L_z\phi-\omega \phi,\quad t\geq0,
\end{align}
starting with an initial guess $\phi(\cdot,0)=\phi_0\in X\backslash\{0\}$. It is clear that the gradient flow \eqref{eq:defocusing-gf} is action-diminishing:
\begin{align}
\frac{d}{dt}S_{\Omega,\omega}(\phi)= 2\mathrm{Re}\int_{\mathbb{R}^d}\frac{\delta S_{\Omega,\omega}(\phi)}{\delta\overline{\phi}}\partial_t\overline{\phi}\,d\mathbf{x}=-2\left\|\frac{\delta S_{\Omega,\omega}(\phi)}{\delta\overline{\phi}}\right\|_{L^2}^2,\quad\forall t\geq0.
\end{align}

Various discretization techniques can be applied to the gradient flow \eqref{eq:defocusing-gf}. Here, we discretize the gradient flow \eqref{eq:defocusing-gf} by the following \emph{backward-forward Euler scheme with stabilization term}:
\begin{align}\label{eq:defocusing-gfbf}
\left\{\begin{aligned}
&\frac{\phi^{n+1}-\phi^n}{\tau}= \frac12\Delta\phi^{n+1}-\alpha_n\phi^{n+1}+\left(\alpha_n-V-\omega-\beta|\phi^n|^{p-1} +\Omega L_z\right)\phi^n,\quad n\geq0, \\
&\phi^0=\phi_0\in X\backslash\{0\},
\end{aligned}\right.
\end{align}
with $\tau>0$ the time step and $\alpha_n\geq0$ the stabilization factor. Obviously in \eqref{eq:defocusing-gfbf}, only a linear elliptic equation with constant coefficient needs to be solved at each time step. Thus, the scheme \eqref{eq:defocusing-gfbf} is very efficient if a fast Poisson solver is available. This is the case when one solves the problem on a bounded computational domain $U\subset\mathbb{R}^d$ with the homogeneous Dirichlet or periodic boundary condition. We shall refer to (\ref{eq:defocusing-gfbf}) as the GF-BF scheme for computing the action ground state. Its implementation is outlined in Algorithm~\ref{algorithm-gfbf}.

\begin{algorithm}[!htp] 
\caption{\bf A GF-BF algorithm.}\label{algorithm-gfbf}

Give $\phi^0=\phi_0\in X\backslash\{0\}$, constant $\tau>0$. Set $n=0$.

\While{stopping criteria are not met}{

  Select a stabilization factor $\alpha_n$

  Solve the linear elliptic equation for $\phi^{n+1}$:
  \[ -\frac{\tau}{2}\Delta\phi^{n+1}+(1+\tau\alpha_n)\phi^{n+1}=\phi^n+\tau\left(\alpha_n-V-\omega-\beta|\phi^n|^{p-1} +\Omega L_z\right)\phi^n \]

  $n:=n+1$

}
\end{algorithm}

Due to the explicit treatment of the nonlinear and rotational terms, it is difficult to establish the exact action-decaying property for the scheme \eqref{eq:defocusing-gfbf}, particularly in the whole space $\bR^d$.  Instead, we can prove \eqref{eq:defocusing-gfbf} on the bounded domain $U$, is unconditionally stable with respect to the \emph{modified action functional}
\begin{align}\label{eq:modaction}
\widetilde{S}_{\Omega,\omega}^n(\varphi):=\int_{U}\left(\frac{1}{2}|\nabla\varphi|^2
+\big(V+\omega+\beta|\phi^n|^{p-1}\big)|\varphi|^2
-\Omega\overline{\varphi}L_z\varphi\right)d\mathbf{x}.
\end{align}
\begin{theorem}\label{lem:gfbf-action-decay}
Let $U$ be a box domain, $\beta>0$ and $V\in L^{\infty}(U)$. If $\phi^n\in H^1(U)\cap L^{\infty}(U)$ and one of the following holds:
\begin{enumerate}[(i)]
  \item $d=1$ and $\alpha_n\geq\frac12\max\big\{0,\mathrm{ess\,sup}_{\mathbf{x}\in U}\big(V(\mathbf{x}) +\omega +\beta|\phi^n(\mathbf{x})|^{p-1}\big)\big\}$;
    \item $d\geq2$ and $\alpha_n\geq\frac12\max\big\{0,\mathrm{ess\,sup}_{\mathbf{x}\in U}\big(V(\mathbf{x})+\frac{|\Omega|^2}{2}(x_1^2+x_2^2) +\omega +\beta|\phi^n(\mathbf{x})|^{p-1}\big)\big\}$;
\end{enumerate}
then the scheme \eqref{eq:defocusing-gfbf} on the spatial domain $U$ with the homogeneous Dirichlet or periodic boundary condition has the following action-decaying property: for any $\tau>0$,
\[ \widetilde{S}_{\Omega,\omega}^n(\phi^{n+1})\leq \widetilde{S}_{\Omega,\omega}^n(\phi^n). \]
\end{theorem}
\begin{proof}
Taking the $L^2$-inner product of \eqref{eq:defocusing-gfbf} with {$-2(\phi^{n+1}-\phi^n)$} and then taking the real part, we get
\begin{align*}
-2\left(\alpha_n+\frac{1}{\tau}\right)\|\phi^{n+1}-\phi^n\|_{L^2}^2
&=\frac12\|\nabla\phi^{n+1}\|_{L^2}^2-\frac12\|\nabla\phi^n\|_{L^2}^2
+\frac12\|\nabla(\phi^{n+1}-\phi^n)\|_{L^2}^2 \\
&\quad +\int_U \big(V+\omega+\beta|\phi^n|^{p-1}\big) \left(|\phi^{n+1}|^2-|\phi^n|^2-|\phi^{n+1}-\phi^n|^2\right) d\mathbf{x} \\
&\quad -\Omega\int_U\left( \overline{\phi^{n+1}}L_z\phi^{n+1}-\overline{\phi^n}L_z\phi^n-\overline{(\phi^{n+1}-\phi^n)}
L_z(\phi^{n+1}-\phi^n)\right) d\mathbf{x} \\
&= \widetilde{S}_{\Omega,\omega}^n(\phi^{n+1})- \widetilde{S}_{\Omega,\omega}^n(\phi^n) -\int_U \big(V+\omega+\beta|\phi^n|^{p-1}\big)|\phi^{n+1}-\phi^n|^2 d\mathbf{x} \\
&\quad +\frac12\|\nabla(\phi^{n+1}-\phi^n)\|_{L^2}^2 +\Omega\int_U\overline{(\phi^{n+1}-\phi^n)}L_z(\phi^{n+1}-\phi^n)d\mathbf{x}.
\end{align*}
Then for case (i), the rotational term vanishes in the above and the assumption on $\alpha_n$ leads to
\[ \widetilde{S}_{\Omega,\omega}^n(\phi^{n+1})- \widetilde{S}_{\Omega,\omega}^n(\phi^n)\leq -\frac12\|\nabla(\phi^{n+1}-\phi^n)\|_{L^2}^2-\frac{2}{\tau}\|\phi^{n+1}-\phi^n\|_{L^2}^2, \]
which shows the assertion. For case (ii), applying Lemma~\ref{lem:rotenergy} with the domain $\mathbb{R}^d$ replaced by $U$ and the assumption on $\alpha_n$, we have
\[ \widetilde{S}_{\Omega,\omega}^n(\phi^{n+1})- \widetilde{S}_{\Omega,\omega}^n(\phi^n)\leq -\frac{2}{\tau}\|\phi^{n+1}-\phi^n\|_{L^2}^2. \]
The proof is completed.
\end{proof}

\subsection{Preconditioned optimization methods}\label{sec3 opt}
Thanks to Theorem ~\ref{thm:exist}, the unconstrained optimization methods can also be applied for the minimization problem
\begin{equation}\label{minimization_model_unconstraint}
\phi_g=\arg \min_{\phi\in X} S_{\Omega,\omega}(\phi).
\end{equation}
These methods will be similar but formally simpler than those presented in Section \ref{subsec opt foc}, and they are very efficient in practical computing. In fact, they are more needed here than for the focusing case, owning to the more complex patterns (vortices) in the ground state solutions in the defocusing case. These will be illustrated later in Section \ref{sec:4}.

The optimization methods here will be based on the iterative scheme of the form
\begin{equation}\label{eq:optimscheme_unconstraint}
 \phi^{n+1}=\phi^{n}+\tau_n d_n, \quad n=0,1,\ldots,
\end{equation}
with $d_n\in X$ a descent direction and $\tau_n>0$ a step length at the $n$-th approximate state $\phi_n$.

The scheme of the unconstrained PBB method is very similar to that of the constrained PBB method proposed in Section \ref{subsec opt foc}, i.e.,
\begin{align}\label{eq:psd_unconstraint}
\phi^{n+1} =\phi^{n}+\tau_n d_n=\phi^n-\tau_n\mathcal{P}g_n,\quad n=0,1,\ldots,
\end{align}
but now $g_n:=\left(-\frac12\Delta+V+\beta|\phi^n|^{p-1}-\Omega L_z+\omega\right)\phi^n$ is the $L^2$-gradient of $S_{\Omega,\omega}$ at $\phi^n$ and no normalization step is required here. Here the preconditioner $\mathcal{P}$ is similarly chosen as
\begin{equation}\label{P-V-ab-defocu}
\mathcal{P}=\mathcal{P}_{V}^{1/2} \mathcal{P}_{\Delta} \mathcal{P}_{V}^{1/2} \;\,\mbox{with}\;\,
\mathcal{P}_{\Delta}=\left(\alpha_{\Delta}-\frac{1}{2}\Delta\right)^{-1} \;\,\mbox{and}\;\,
\mathcal{P}_{V}=\left(\alpha_{V}+V+\beta|\phi^n|^{p-1} \right)^{-1}.
\end{equation}
Based on our numerical experience, a simple choice $\alpha_{\Delta}=\alpha_{V}=(h^{-2}+L^d)/2$ is suggested in practice, with $h$ the mesh size of the computation domain $\left[-L,L \right]^d$.
{Denoting $s_{n-1}:=\phi^n-\phi^{n-1}$ and $y_{n-1}:=\mathcal{P}\left(g_n-g_{n-1}\right)$, the BB step length reads as
\begin{align}
\tau^{\mathrm{BB}}_n:= \frac{\left|\langle y_{n-1},s_{n-1}\rangle_{L^2}\right|}{\langle y_{n-1},y_{n-1} \rangle_{L^2}},\quad n\geq1.
\end{align}}
% Denoting $s_{n-1}:=\phi^n-\phi^{n-1}$ and $y_{n-1}:=d_n-d_{n-1}=\mathcal{P}\left(g_n-g_{n-1}\right)$, the BB step lengths read as
% \begin{align}\label{eq:BBsteps_unconstraint}
% \tau^{\mathrm{BB1}}_n:= \frac{\left|\langle y_{n-1},s_{n-1}\rangle_{L^2}\right|}{\langle y_{n-1},y_{n-1} \rangle_{L^2}}\quad\mbox{and}\quad
% \tau^{\mathrm{BB2}}_n:= \frac{\langle s_{n-1},s_{n-1} \rangle_{L^2}}{\left|\langle y_{n-1},s_{n-1}\rangle_{L^2}\right|},\quad n\geq1.
% \end{align}
The practical unconstrained PBB algorithm framework is outlined in Algorithm \ref{algorithm3}.
\begin{algorithm}[!htp]
\caption{\bf A unconstrained PBB algorithm.}\label{algorithm3}

Give $\phi^0=\phi_0$, constants $0<\tau_{\min}<\tau_0<\tau_{\max}$. Set $n=0$.

\While{stopping criteria are not met}{

  $g_n=\left(-\frac12\Delta+V+\beta|\phi^n|^{p-1}-\Omega L_z+\omega\right)\phi^n$

  $\tau_n=\begin{cases}
   \tau_0, &\mbox{if $n=0$ or }\langle y_{n-1},s_{n-1}\rangle_{L^2}=0, \\
   \max\{\min\{{\tau^{\mathrm{BB}}_n}, \tau_{\max}\},\tau_{\min}\}, & \mbox{otherwise},
   \end{cases}$

  $\phi^{n+1}=\phi^n-\tau_{n}\mathcal{P} g_n$

  $n:=n+1$

}
\end{algorithm}

The PCG method for the unconstrained problem \eqref{minimization_model_unconstraint} can also be proposed by taking in \eqref{eq:optimscheme_unconstraint}:
\begin{equation}\label{pcg_direction_unconstraint}
d_n=\begin{cases} -\mathcal{P} g_n, & n=0,\\
-\mathcal{P} g_n+\beta_nd_{n-1},  & n\geq 1,
\end{cases}
\end{equation}
with $\mathcal{P}$ defined in \eqref{P-V-ab-defocu} and $\beta_n$ given by the Polak-Ribi\`{e}re-Polyak formula \cite{Yuan}
\begin{equation}\label{pcg-beta_unconstraint}
 \beta_n=\max\{\beta_n^{\mathrm{PRP}},0\},\quad
 \beta_n^{\mathrm{PRP}}=\frac{ \mathrm{Re} \big\langle g_n-g_{n-1}, \mathcal{P} g_n \big \rangle_{L^2} } { \big \langle g_{n-1}, \mathcal{P} g_{n-1} \big \rangle_{L^2} }.
\end{equation}
And, the optimal step size $$\tau_n^{opt}=\arg \min_{\tau>0} S_{\Omega,\omega} \left(\phi^n+\tau d_n\right)$$
can be efficiently obtained again by Brent's method \cite{Brent1,Brent2}. Now, according to \eqref{eq:optimscheme_unconstraint}, \eqref{pcg_direction_unconstraint}-\eqref{pcg-beta_unconstraint}, the unconstrained PCG method is summarized in Algorithm \ref{algorithm4}.
\begin{algorithm}[!htp]
\caption{\bf A unconstrained PCG method.}\label{algorithm4}
Give $\phi^0=\phi_0$. Set $n=0$.

\While{stopping criteria are not met}{

$g_n=\left(-\frac12\Delta+V+\beta|\phi^n|^{p-1}-\Omega L_z+\omega\right)\phi^n$

$d_n=\begin{cases} -\mathcal{P} g_n, & n=0,\\
-\mathcal{P} g_n+\beta_nd_{n-1},  & n\geq 1,
\end{cases}$ \ with \ $\beta_n=\max\{\beta_n^{\mathrm{PRP}},0\}$

$\tau_n^{opt}=\arg \min_{\tau>0} S_{\Omega,\omega} \left(\phi^n+\tau d_n\right)$

$\phi^{n+1}=\phi^n+\tau_{n}^{opt}  d_n$

$n:=n+1$
}
\end{algorithm}

The stopping criterion in Algorithms~\ref{algorithm3} and \ref{algorithm4} can be taken as one of the following:
\begin{subequations}\label{residual defoc}
\begin{align}
&r_{err}^{n,\infty}:=\left\|\left(-\frac12\Delta+V+\beta|\phi^n|^{p-1}-\Omega L_z+\omega\right)\phi^n\right\|_\infty   \leq \varepsilon,\\
&\mathcal{E}_{err}^n:= \left | S_{\Omega,\omega}(\phi^{n+1})-S_{\Omega,\omega}(\phi^n) \right| \leq \varepsilon.\end{align}
\end{subequations}

\section{Numerical experiments}\label{sec:4}
In this section, we carry out numerical experiments to test the performance of the proposed methods and explore the features of the action ground states. We shall present separately for the focusing case and the defocusing case.

\subsection{Focusing case}\label{sec:4-1}
We begin with the focusing case of the RNLS (\ref{model}) and compute the action ground state of (\ref{action}) by techniques introduced in Section \ref{sec:2}. We consider two examples in the following to illustrate the accuracy and efficiency of the proposed methods.

\begin{example}\label{example-1}
Firstly, we test the gradient flow approach and verify our theoretical findings in Section \ref{sec:2}. To do so,  we take $d=1$, $V=0$, $\beta=-1$ and $p=3$ in \eqref{model0} which gives
\begin{align}\label{eq:cubicNLS-1d}
	\frac12\partial_{xx}\phi(x)+|\phi(x)|^2\phi(x)=\omega\phi(x),\quad x\in\mathbb{R}.
\end{align}
The unique positive ground state solution (up to a translation) of \eqref{eq:cubicNLS-1d} is available   explicitly:
\begin{align}\label{eq:sech}
	\phi_g(x)=\sqrt{2\omega}\,\mathrm{sech}\big(\sqrt{2\omega}x\big),\quad x\in\mathbb{R}.
\end{align}
The corresponding value of the action functional is given by $S(\phi_g):=\frac12\|\phi_g\|^4_{L^4}=\frac{4\sqrt{2}}{3}\omega^{3/2}$. We fix $\omega=1$ here. Based on Theorem \ref{thm:minQ}, the equivalent constrained minimization problem now reads
\begin{align}\label{eq:cubicNLS-1d-minQ}
	\mbox{minimize}\quad  Q(u)=\int_{\mathbb{R}}\left(\frac12|\partial_{x}u(x)|^2+\omega|u(x)|^2\right)dx \quad\mbox{subject to}\quad \|u\|_{L^4}^4=1,
\end{align}
and by the GFALM-BF scheme \eqref{eq:gfalmbf-minQ} (i.e., Algorithm \ref{algorithm-gfalmbf}) we compute numerically the ground state of
\eqref{eq:cubicNLS-1d}.
\end{example}

\begin{figure}[!htp]
\psfig{figure=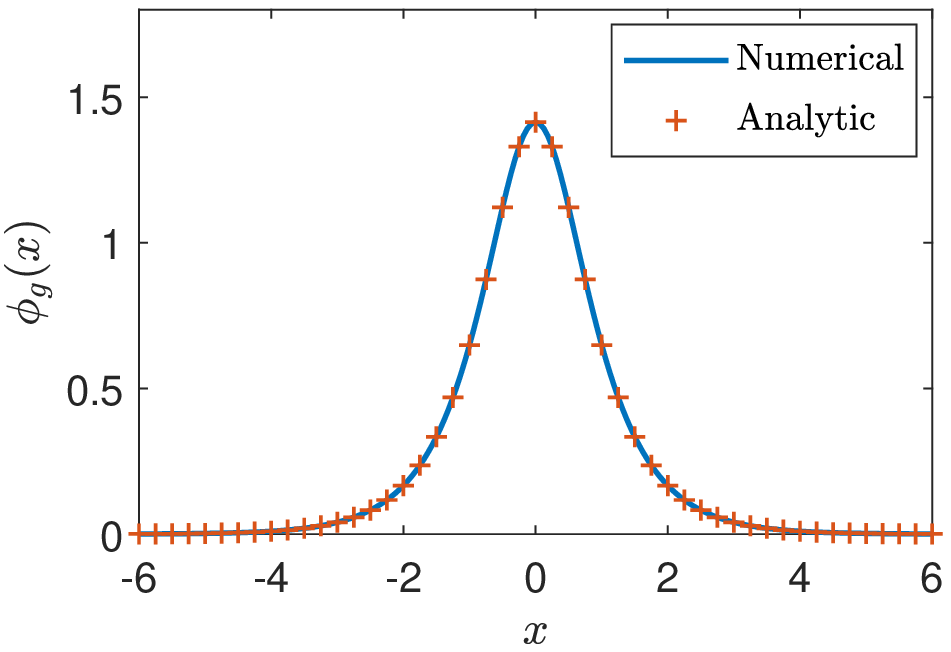,width=2.9in,height=2.in,angle=0}
\psfig{figure=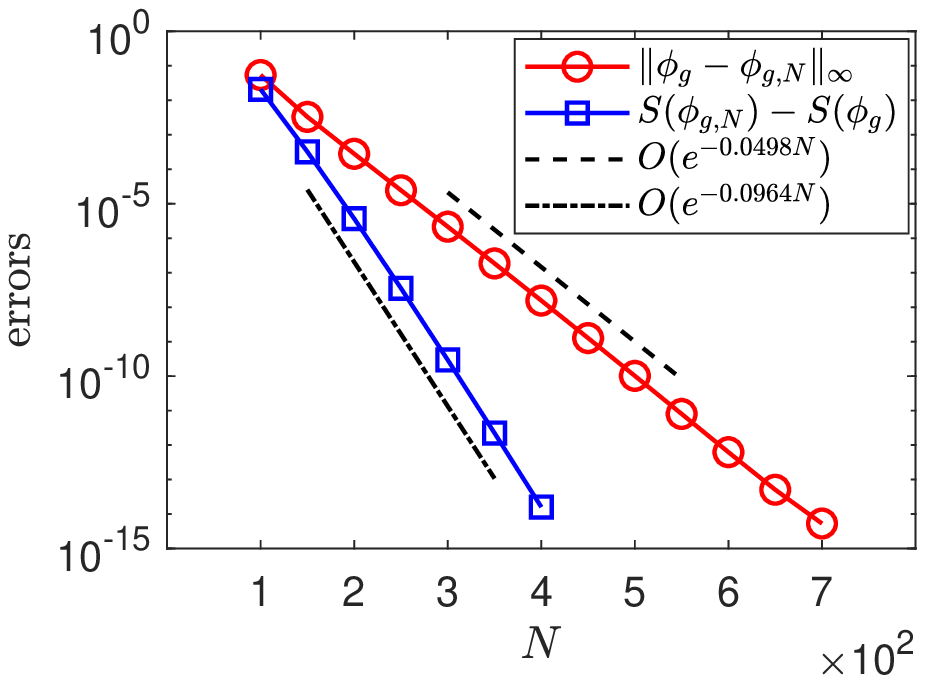,width=2.9in,height=2.in,angle=0}
  \caption{Accuracy of GFALM-BF in Example~\ref{example-1}: profiles of the numerical solution $\phi_{g,N}$ with $N=2^{10}$ and the analytic solution $\phi_g$ (left); error of the numerical solution and the action functional with respect to $N$ (right).}
  \label{fig:sech-omega1}
\end{figure}
\begin{figure}[!htp]
  \centering
\psfig{figure=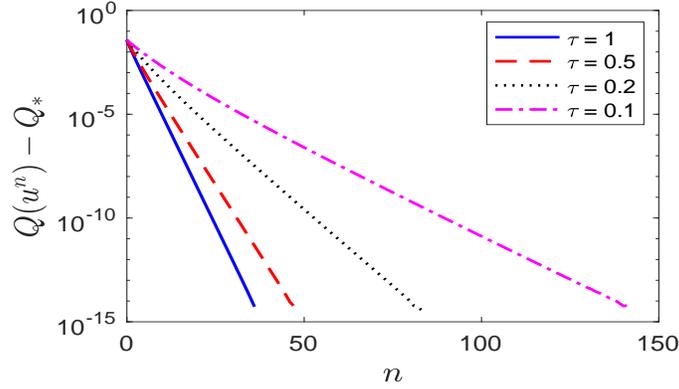,width=3.5in,height=2.in,angle=0}
  \\
  \caption{The change of $Q(u^n)-Q_*$ (in logarithmic scale) with respect to the number of iterations $n$ in GFALM-BF under different $\tau$ in Example~\ref{example-1}.}
  \label{fig:sech-omega1-Qdecay}
\end{figure}

We implement Algorithm \ref{algorithm-gfalmbf} on the computational domain $U=(-32,32)$ with the spatial Fourier pseudospectral discretization \cite{BaoCai,Trefethen} with $N$ discrete Fourier modes. The stabilization factor is chosen as $\alpha_n=\frac12\max_{0\leq j\leq N-1}\{0,\omega-Q(u^n)|u^n(x_j)|^2\}+1$ (by Theorem \ref{thm:gfalmbf-decay}), where $x_j=-32+jh$, $j=0,\ldots,N$, with $h=64/N$. The time step length is taken by default as $\tau=1$, and the initial data is chosen as $u_0(x)=(\pi/2)^{-1/8}\fe^{-x^2/2}$. We stop the iteration in Algorithm \ref{algorithm-gfalmbf} and adopt $u^n$ as the solution of \eqref{eq:cubicNLS-1d-minQ} if the maximal residual of the Euler-Lagrange equation of \eqref{eq:cubicNLS-1d-minQ} is less than $10^{-14}$, i.e.,
\begin{align*}
	\max_{0\leq j\leq N-1}\left|-\frac12 \partial_{xx}u^n(x_j)+\omega u^n(x_j)-Q(u^n)|u^n(x_j)|^2u^n(x_j)\right|<10^{-14}.
\end{align*}
Then by (\ref{Q relation}), the numerical ground state to \eqref{eq:cubicNLS-1d} is computed as $\phi_{g,N}(x):=\sqrt{Q(u^n)}\, u^n(x)$.

In Fig.~\ref{fig:sech-omega1}, we plot the profiles of the analytic solution (\ref{eq:sech}) and the numerical solution $\phi_{g,N}$ with $N=2^{10}$ in the left subplot. In the right subplot, we show the error of the numerical ground state  $\|\phi_g-\phi_{g,N}\|_{\infty}:=\max_{0\leq j\leq N}|\phi_g(x_j)-\phi_{g,N}(x_j)|$ and the error of the action functional $S(\phi_{g,N})-S(\phi_g)$ with respect to $N$. From the results, we can clearly observe the effectiveness of the GFALM-BF method in Algorithm \ref{algorithm-gfalmbf} and the spectral accuracy of the spatial discretization.

Then, we consider the evolution of the quadratic energy $Q$ defined in (\ref{Q def}) under the GFALM-BF scheme \eqref{eq:gfalmbf-minQ}. The difference $Q(u^n)-Q_*$ is plotted under several time steps in Fig.~\ref{fig:sech-omega1-Qdecay}, where $Q_*=Q(u_*)=\sqrt{8\omega\sqrt{2\omega}/3}$ and $u_*=\phi_g/\|\phi_g\|_{L^4}$. The result clearly illustrates the decay of  the quadratic energy in GFALM-BF, which verifies our theoretical result in Theorem~\ref{thm:gfalmbf-decay}.

\begin{example}\label{example-2}
Next, we illustrate the efficiency of the proposed methods by considering a two-dimensional example. We take $d=2$, $p=3$, $\beta=-1$, $\omega=1$ and $V(\bx)=\frac{1}{2}(\gamma_1^2x_1^2+\gamma_2^2x_2^2)$ with $\gamma_1=\gamma_2=1$ in (\ref{model}). The action ground state (\ref{phi_g-def}) will be computed by the GFALM-BF scheme and by the two optimization methods PBB and PCG in Section~\ref{sec:2}. We compare their efficiency from the normal to fast rotating regime by taking $\Omega=0.3,0.5,0.7,0.8,0.9$ (with $|\Omega|<\min\{\gamma_1,\gamma_2\}$ satisfied).
\end{example}

We fix the computation domain $U=(-4,4)^2$ with mesh size $h=1/16$. We solve the problem by the three proposed methods in Section \ref{sec:2}, i.e., the GFALM-BF in Algorithm \ref{algorithm-gfalmbf} with $\tau=0.1$, the PBB in Algorithm \ref{algorithm1} and the PCG in Algorithm \ref{algorithm2} with the preconditioner (\ref{P-V-ab}). The stopping criterion $\mathcal{E}_{err}^n \leq 10^{-14}$ (\ref{energy stop}) is applied for all methods. In this example, we do not have the exact formula of the action ground state. Also, note that the proposed methods may end up at local minima which depend on the choice of initial data. Inspired by \cite{ALT2017,BaoCai,Danaila1,Wen}, we shall consider the following six types of functions
\begin{subequations}
\begin{align}
&(a)\ \phi_a(\mathbf{x})=\sqrt{\gamma_1\gamma_2/\pi}\, \fe^{-\left(\gamma_1 x_1^2+\gamma_2 x_2^2\right)/2},\quad\quad\ \, (b)\  \phi_b(\mathbf{x})=(x_1+i x_2) \phi_a(\mathbf{x}),\label{initial-T1}\\
&(c)\  \phi_c(\mathbf{x})=(x_1+i x_2)^4 \phi_a(\mathbf{x}),\quad\quad\quad\quad\quad\quad (d)\ \phi_d(\mathbf{x})=\left( \phi_a(\mathbf{x}) +\phi_b(\mathbf{x}) \right)/2, \label{initial-T2} \\
&(e)\ \phi_e(\mathbf{x})=(1-\Omega) \phi_a(\mathbf{x}) +\Omega\phi_b(\mathbf{x}),\quad\quad\quad\, (f)\ \phi_f(\mathbf{x})=\Omega \phi_a(\mathbf{x}) +(1-\Omega)\phi_b(\mathbf{x}), \label{initial-T3}
\end{align}
\end{subequations}
and use their $L^{p+1}$-normalizations as the initial data for computing the action ground states.

\begin{figure}[!htp]
\centerline{
\psfig{figure=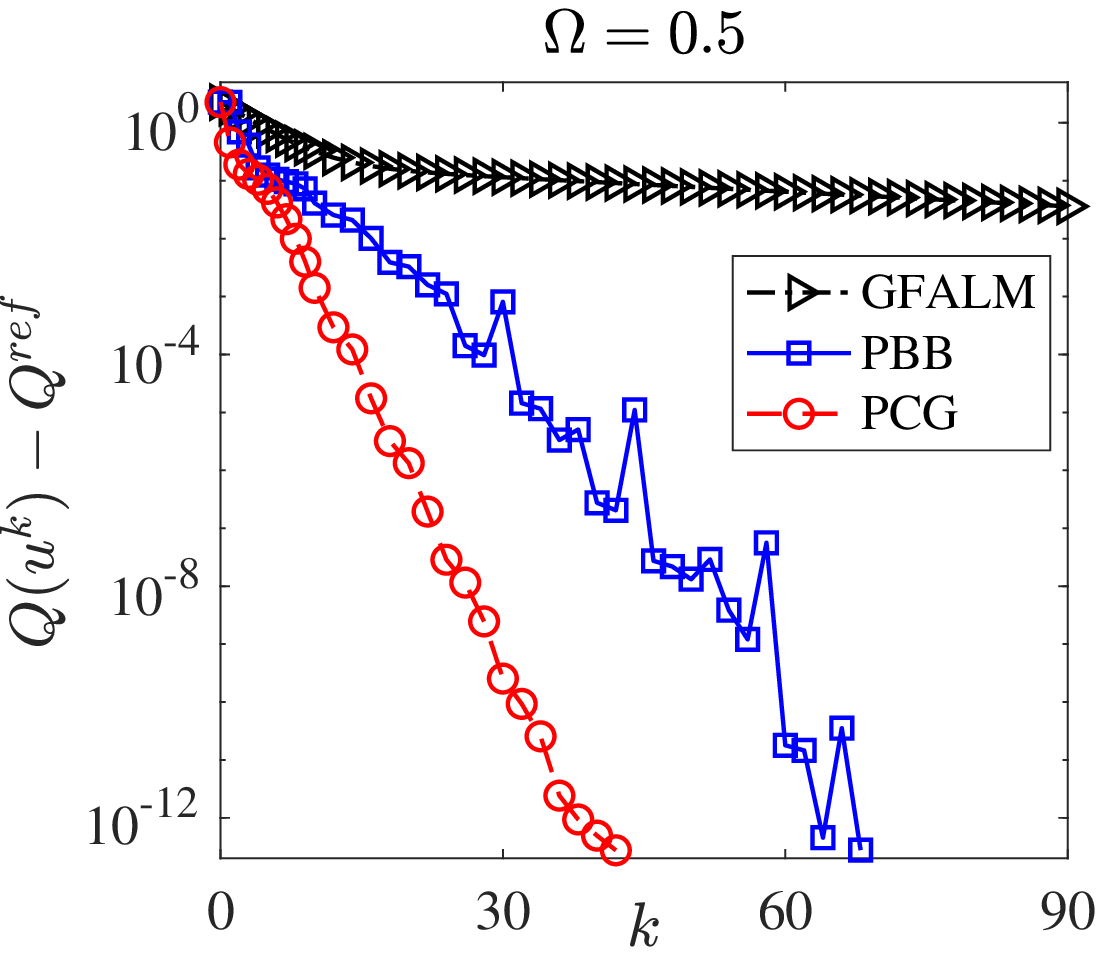,width=2.8in,height=2.in,angle=0}\quad
\psfig{figure=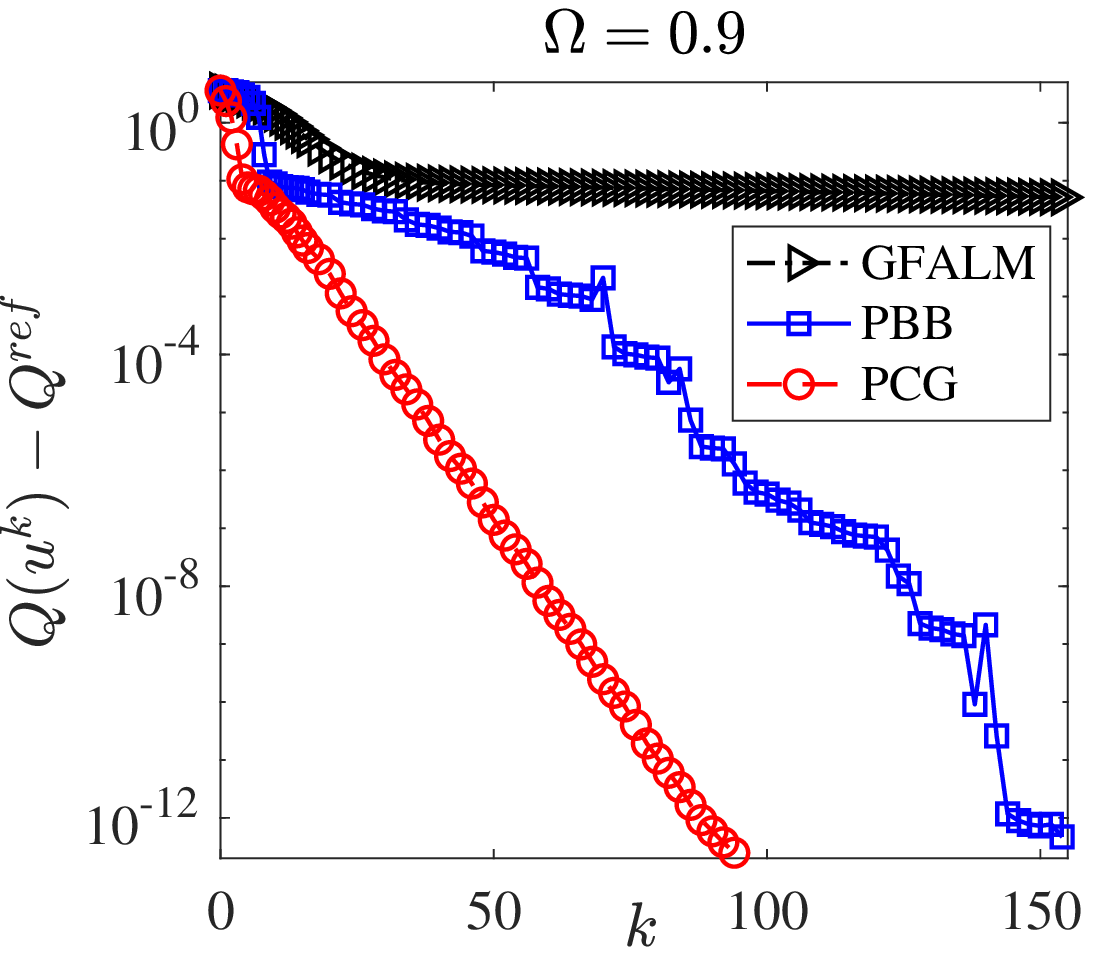,width=2.8in,height=2.in,angle=0}
}
\caption{The change of $Q(u^k)-Q^{ref}$ (in logarithmic scale) with respect to the number of iterations $k$ under $\Omega=0.5$ (left) and $\Omega=0.9$ (right) in Example \ref{example-2}, where $Q^{ref}=3.753300166998$ is the reference value of $Q(u_*)$ for both cases.}\label{fig_Qens}
\end{figure}
\begin{figure}[!ht]
\centerline{
\psfig{figure=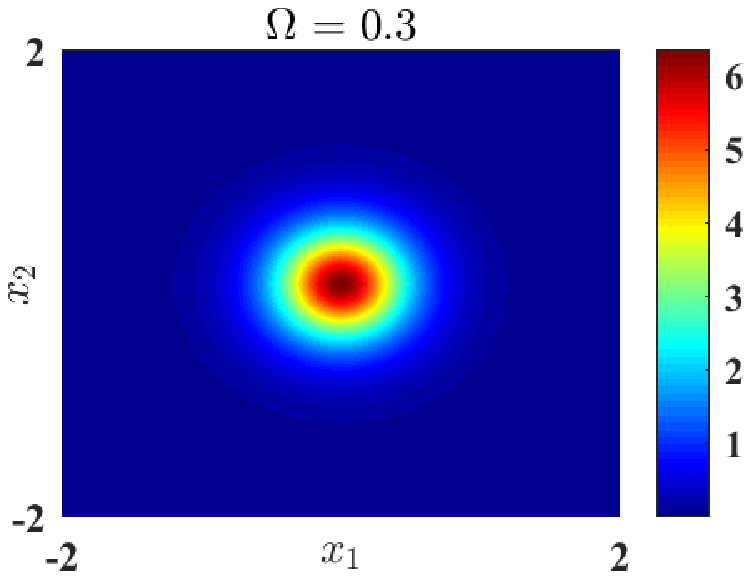,width=1.9in,height=1.7in,angle=0}
\psfig{figure=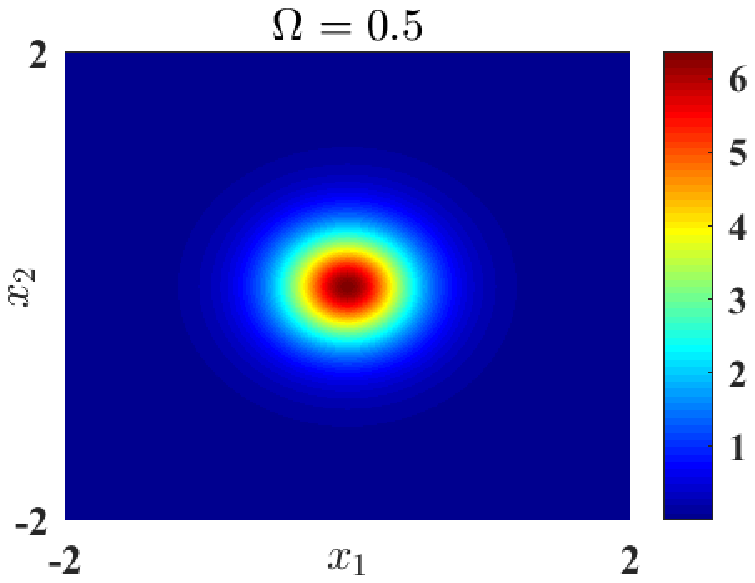,width=1.9in,height=1.7in,angle=0}
\psfig{figure=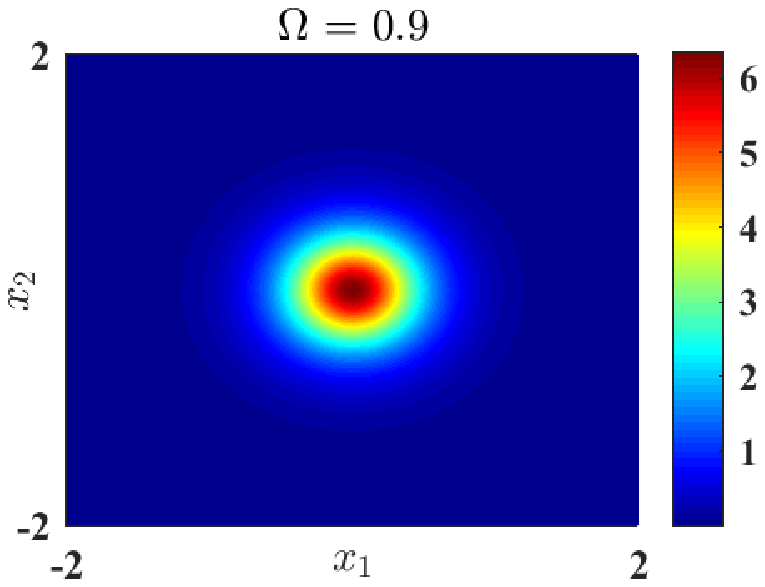,width=1.9in,height=1.7in,angle=0}
}
 \caption{Contour plots of $|\phi_g|^2$ for different
$\Omega$ in Example \ref{example-2}.}\label{fig_gs_focusing}
\end{figure}
\begin{table}[!htp]
\centering\small
\caption{Comparison of the GFALM, PBB and PCG  in Example \ref{example-2}.}
\label{tab:foC_a}
\begin{tabular}{ccccccc}
\hline
Method & $\Omega$ & $iter$ & CPUs & $S_{\Omega,\omega}(\phi_g)$ & $\mathcal{E}_{err}^g$ & $r_{err}^{g,\infty}$\\
\hline
\multirow{5}{*}{GFALM}
%& 0 & - & - & - & -& -   \\
& 0.3 & 1056 & 2.28 &7.04363107 & 9.77E-15&9.29E-07    \\
& 0.5 & 1342 & 2.85 &7.04363107 & 9.33E-15& 9.20E-07   \\
& 0.7 & 2006& 3.95 & 7.04363107 & 7.99E-15& 8.72E-07   \\
%& 0.8 & 2388& 4.89 & 7.04357030 & 9.77E-15& 9.42E-07   \\
& 0.9 & 5128 &9.97 & 7.04363107 &4.56E-15 & 8.86E-07  \\
\hline
\multirow{5}{*}{PBB}
%& 0 & - & - & - & -& -   \\
& 0.3 & 49 & 0.91 &7.04363107 & 4.44E-16& 6.92E-06   \\
& 0.5 & 70 & 0.98 &7.04363107 & 7.11E-15& 7.24E-07  \\
& 0.7 & 90 &1.05 &7.04363107 &3.11E-16& 3.04E-07  \\
%& 0.8 & 126 &1.15 &7.04357031 &5.32E-15& 3.48E-07  \\
& 0.9 & 156 & 1.31 &7.04363107 &6.66E-15& 5.18E-07  \\
\hline
\multirow{5}{*}{PCG}
%& 0 & - & - & - & -& -   \\
& 0.3 & 43 & 3.53 &7.04363107 & 4.44E-15& 3.33E-07   \\
& 0.5 & 49 & 3.95 &7.04363107 & 1.77E-15& 6.54E-07  \\
& 0.7 & 63 &4.50 &7.04363107&4.00E-15& 7.56E-07  \\
%& 0.8 & 74 &6.23 &7.04357030 &4.44E-15& 1.75E-06  \\
& 0.9 & 99 & 6.88 & 7.04363107 &8.88E-15& 1.53E-06  \\
\hline
\end{tabular}
\end{table}

Our first numerical observation is that any types of initial data in \eqref{initial-T1}-\eqref{initial-T3} converge very fast to the same state in the GFALM-BF, PBB and PCG methods. For simplicity, here we only show the results with $\phi_0=\phi_e/{\|\phi_e\|_{L^4}}$. Fig. \ref{fig_Qens} shows the decrease of $Q(u^k)$ (\ref{Q def}) with respect to the number of iterations $k$ in the three methods. Tab. \ref{tab:foC_a} presents the total number of iterations (iter), the computational time in seconds (CPUs\footnote{Programmed sequentially in MATLAB and run on a MacBook 2.4 GHz Intel Core i5.}), (\ref{r_err}) and (\ref{energy stop}) of the methods under different $\Omega$.

It can be seen that the GFALM-BF, PBB and PCG methods are all able to get the action ground states  accurately. The two optimization methods are more efficient than the gradient flow method in practice. Between the two, PCG takes the least number of iterations, whereas PBB takes the least CPUs to converge. Note that the PCG method uses Brent's method to get the adaptive step size, which may  take more CPUs than that of the BB step length strategy applied in the PBB method. Moreover, we observe from Tab. \ref{tab:foC_a} that the value of $S_{\Omega,\omega}(\phi_g)$ does not change with respect to the rotational speed $\Omega$. This is because the obtained ground state functions here are positive (which verifies the result in Lemma \ref{lem:focusing-min-existence}), and the positive ground state never contributes to the rotation part in \eqref{action}. The profiles of the action ground states as shown in Fig. \ref{fig_gs_focusing} are all Gaussian-like waves in the focusing case.

\subsection{Defocusing case}\label{sec:4-2}
Now we consider the defocusing case of the RNLS (\ref{model}). We shall use two examples to illustrate the performance of the methods introduced in Section \ref{sec:3}.

\begin{example}\label{example-3}
Similarly as before, we first consider an one-dimensional example to justify the theoretical results in Section \ref{sec:3}. Take $d=1$, $\beta=1$, $p=3$ in \eqref{model0}, and set $V(x)=0$ for $0<x<L$ with a given $L>0$ and $V(x)=\infty$ otherwise. Then, the benchmark problem reads
	\begin{align}\label{eq:defocusing-cubicNLS-1d}
		\frac12\partial_{xx}\phi(x)=|\phi(x)|^2\phi(x)+\omega\phi(x),\quad 0<x<L,
	\end{align}
	with homogeneous Dirichlet boundary conditions. The analytic ground state solution is expressed by the Jacobi elliptic function $\mathrm{sn}(\cdot,\cdot)$ (see, e.g., \cite{CCR2000PRA} for more details):
	\begin{align}\label{eq:sn}
		\phi_g(x)= \frac{2 k K(k)}{L} \, \mathrm{sn}\left(\frac{2K(k)x}{L},k\right),\quad x\in[0,L],
	\end{align}
	where $K(k)=\int_0^{\frac{\pi}{2}}\left(1-k^2\sin^2\theta\right)^{-1/2}d\theta$ is the complete elliptic integral of the first kind and the modulus $k\in[0,1]$ is determined by the equation
$2(1+k^2)K(k)^2+\omega L^2=0.$
	We fix $\omega=-10,\,L=1$, and numerically compute the action ground state by the GF-BF scheme \eqref{eq:defocusing-gfbf} (i.e., Algorithm~\ref{algorithm-gfbf}) that minimizes the action functional $S(\phi):=\int_0^L\left(\frac12|\partial_{x}\phi(x)|+\omega|\phi(x)|+\frac12|\phi(x)|^4\right)dx$.
\end{example}

\begin{figure}[!htp]
  \centering
\psfig{figure=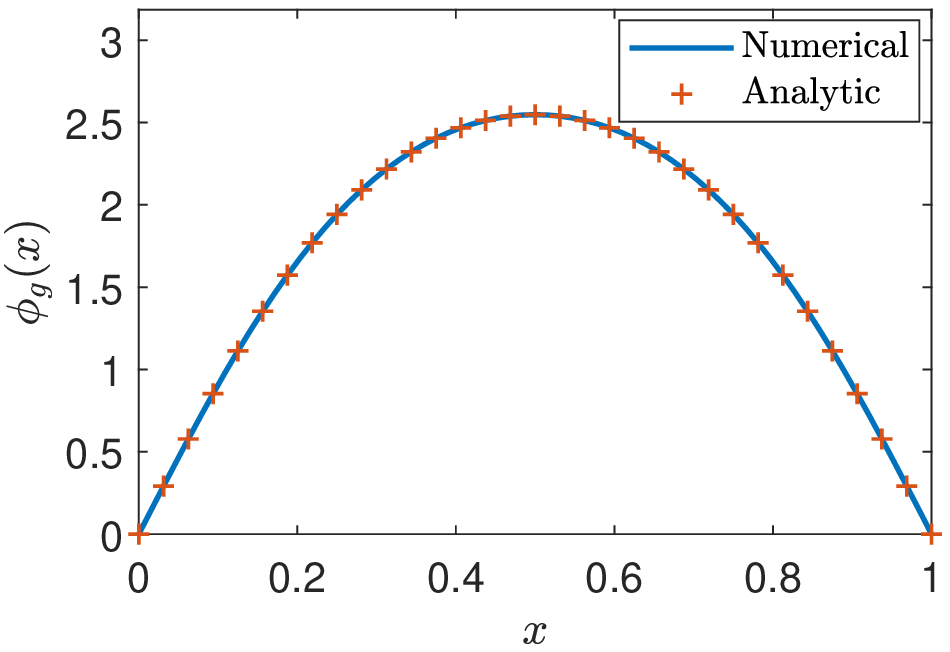,width=2.9in,height=2.in,angle=0}
\psfig{figure=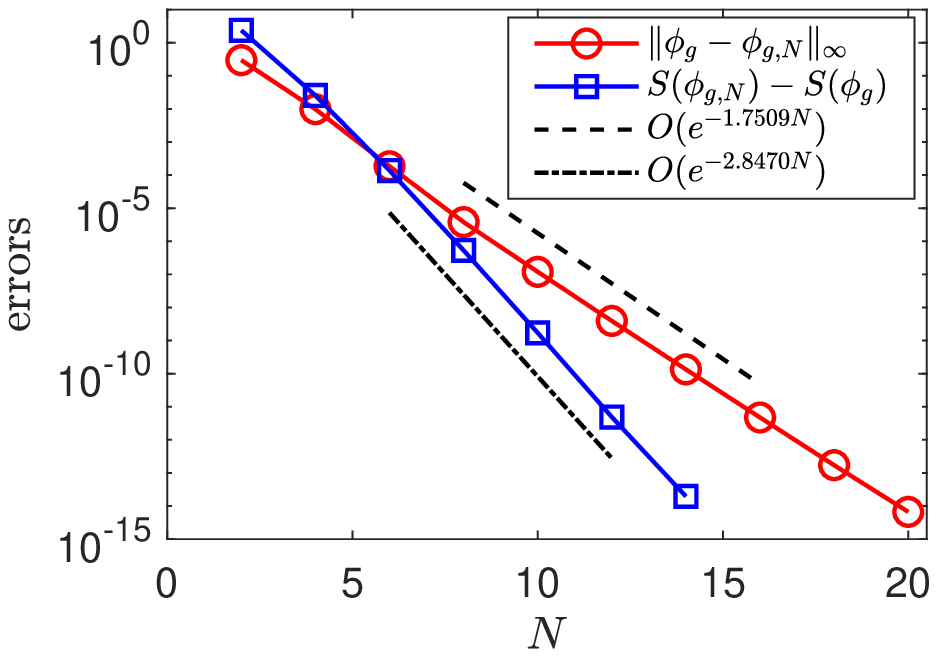,width=2.9in,height=2.in,angle=0}
  \caption{Accuracy of GF-BF in Example~\ref{example-3}: profiles of  the numerical solution $\phi_{g,N}$ with $N=2^{6}$ and the analytic solution $\phi_g$ (left); error of the numerical solution and the action functional with respect to $N$ (right).}
  \label{fig:sn-omega-10}
\end{figure}
\begin{figure}[!htp]
  \centering
\psfig{figure=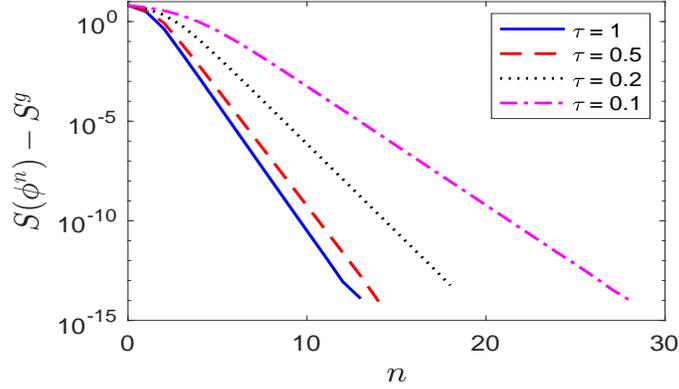,width=3.5in,height=2.in,angle=0}
  \\
  \caption{The change of $S(\phi^n)-S^g$ (in logarithmic scale) with respect to the number of iterations $n$ in GF-BF under different $\tau$ in Example~\ref{example-3}.}
  \label{fig:sn-omega-10-Sdecay}
\end{figure}

We implement Algorithm~\ref{algorithm-gfbf} in space by using the sine pseudospectral discretization \cite{BaoCai} with $N$ discrete sine modes. The time step  is taken as $\tau=1$ and the stabilization factor is chosen as $\alpha_n=\frac12\max_{1\leq j\leq N-1}\{0,\omega+|\phi^n(x_j)|^2\}+1$ (by Theorem \ref{lem:gfbf-action-decay}), where $x_j=jL/N$, $j=0,1,\ldots,N$. The initial data is chosen as $\phi_0(x)=\sin(\pi x)$. We stop the iteration in Algorithm~\ref{algorithm-gfbf} and adopt $\phi_{g,N}:=\phi^n$ as the numerical solution if the maximal residual of the equation \eqref{eq:defocusing-cubicNLS-1d} is less than $10^{-14}$, i.e.,
\begin{align*}
  \max_{1\leq j\leq N-1}\left|-\frac12 \partial_{xx}\phi^n(x_j) +\omega \phi^n(x_j)+|\phi^n(x_j)|^2\phi^n(x_j)\right|<10^{-14}.
\end{align*}

In Fig.~\ref{fig:sn-omega-10},  the left subplot presents the profiles of the analytic solution $\phi_g$ and the numerical solution $\phi_{g,N}$ with $N=2^{6}$. The right subplot shows the error of the numerical solution $\|\phi_g-\phi_{g,N}\|_{\infty}$ and the error of the action functional value $S(\phi_{g,N})-S(\phi_g)$ against $N$. Clearly, the results illustrate the effectiveness and the spatial spectral convergence of the GF-BF method. Moreover, in Fig.~\ref{fig:sn-omega-10-Sdecay} we plot the evolution of $S(\phi^n)-S^g$ in GF-BF with different time steps, where $S^g=S(\phi_g)\approx-8.78043500596719$ is obtained based on \eqref{eq:sn} and accurate quadratures. The result shows the decay of the original action functional in the GF-BF scheme \eqref{eq:defocusing-gfbf}, though we are only able to establish in  Theorem~\ref{lem:gfbf-action-decay} the decay of a modified action. The improvement of the theoretical analysis will be our future work.

\begin{example}\label{example-4} Now we consider a two-dimensional example by taking $d=2$, $p=3$, $\beta=100$, $\omega=-10$, and $V(\bx)=\frac{1}{2}(\gamma_1^2x_1^2+\gamma_2^2x_2^2)$ with  $\gamma_1=\gamma_2=1$ in (\ref{model}). The action ground state will be computed by the GF-BF, PBB and PCG methods in Section \ref{sec3 opt}.
\end{example}

\begin{table}[!htp]
\centering\small
\caption{The action values at ground states (underlined with blue colour) or local minima obtained by GF-BF, PBB, PCG with  \eqref{initial-T1}-\eqref{initial-T3} in Example \ref{example-4}.}
\label{tab:initial_eng}
\begin{tabular}{cccccccc}
\hline
Method & $\Omega$ & $(a)$ & $(b)$ & $(c)$ & $(d)$ & $(e)$ & $(f)$ \\
\hline
\multirow{5}{*}{GF-BF}
& 0.2 & {\color{blue}\underline{-10.0938}} & -9.8055 & -10.0938 & -10.0938& -10.0938 &-10.0938  \\
& 0.3 & -10.0938 &{\color{blue}\underline{-10.1185}} & -10.0938 & -10.1185& -10.0938 &-10.1185   \\
& 0.5 & -10.0938 & -10.7638 & {\color{blue}\underline{-11.2055}}& -10.7638 &-10.7638 &-10.7638  \\
& 0.7 & -10.0938 & -11.4354 & -15.1161 &{\color{blue} \underline{-15.1408}}& -15.1133 &-15.1360  \\
& 0.8 & -10.0938 & -11.7811 &  -20.5439 & -20.5944& -20.5944 &{\color{blue}\underline{-20.6026} } \\
& 0.9 & -37.5563 & -37.5563 & -37.5733 & -37.5733& {\color{blue}\underline{-37.5733}} &-37.5733  \\
\hline
\multirow{5}{*}{PBB}
& 0.2 &  {\color{blue}\underline{-10.0938} }& -10.0938 & -10.0938 & -10.0938& -10.0938 &-10.0938 \\
& 0.3 & -10.0938 & {\color{blue}\underline{-10.1185}} & -10.0938 &-10.1185&  -10.0938 &-10.1185   \\
& 0.5 & -10.0938 & -10.7638 & {\color{blue}\underline{-11.2055}}& -10.7638&-10.7638&-10.7638  \\
& 0.7 & -10.0938 &  {\color{blue}\underline{-15.1408} }& -15.1360 & -15.1219& -15.1219 &-15.1245   \\
& 0.8 & -20.5991 & -20.5946 & -20.5946 & -20.5951& -20.5946 &{\color{blue}\underline{-20.6026} } \\
& 0.9 & -37.5733 & -37.5733 & -37.5733 & -37.5563&  {\color{blue}\underline{-37.5733}} &-37.5733  \\
\hline
\multirow{5}{*}{PCG}
& 0.2 & {\color{blue}\underline{-10.0938} }& -9.8055 & -10.0938 & -10.0938 & -10.0938 &-10.0938  \\
& 0.3 & -10.0938 & {\color{blue}\underline{-10.1185} }& -10.0938 & -10.1185& -10.1185 &-10.1185   \\
& 0.5 &-10.0938& -10.7638 &{\color{blue} \underline{-11.2055}} & -10.7638 &-10.7638& -10.7638  \\
& 0.7 & -10.0938 & -11.4354 & -15.1219 & -15.1408 &{\color{blue} \underline{-15.1408} }&-15.1360   \\
& 0.8 & -10.0938 & -11.7811 &  -20.5949 & -20.5991 & -20.5944 &{\color{blue} \underline{-20.6026}}  \\
& 0.9 & -37.5563 & -37.5733 & -37.5733 & -37.5733& {\color{blue}\underline{-37.5733}}&-37.5563  \\
\hline
\end{tabular}
\end{table}
\begin{figure}[!htp]
\centerline{
\psfig{figure=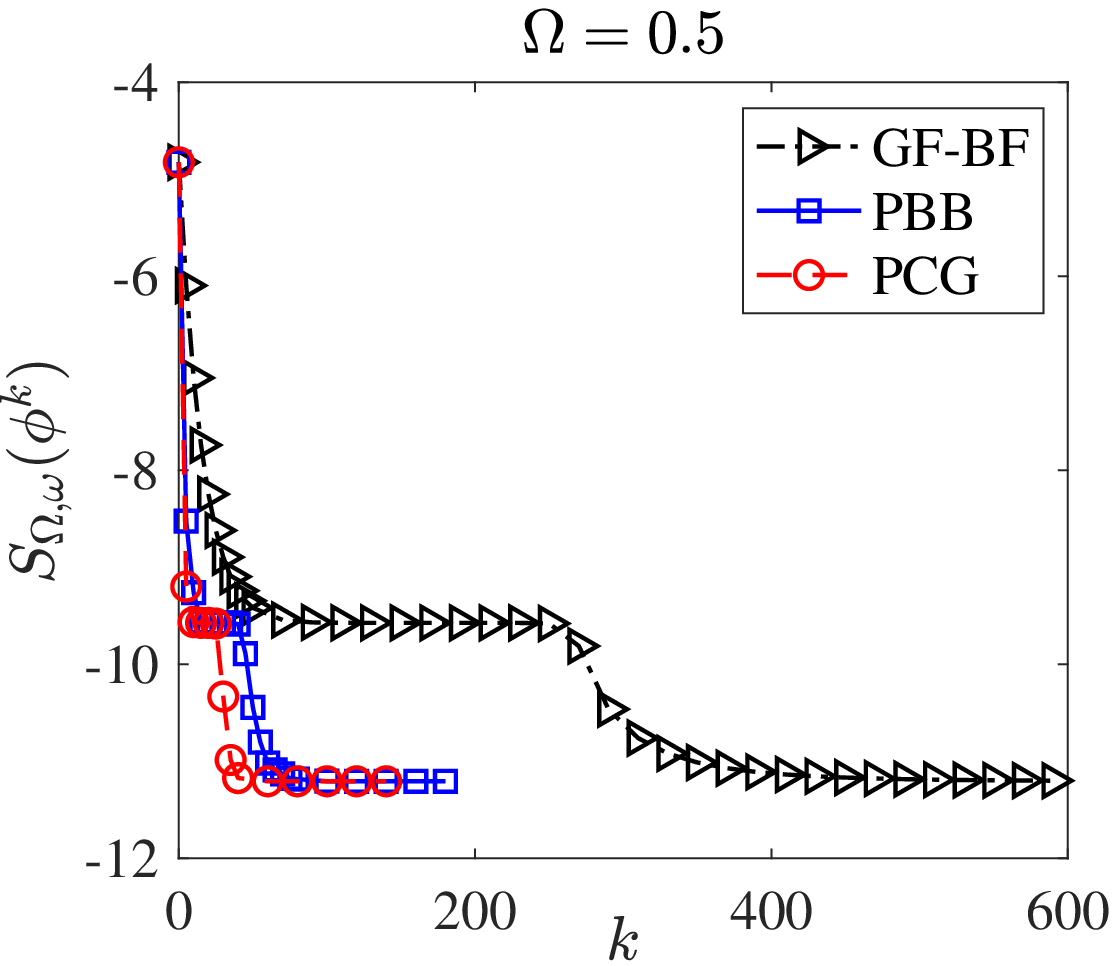,width=2.8in,height=2.in,angle=0}\quad
\psfig{figure=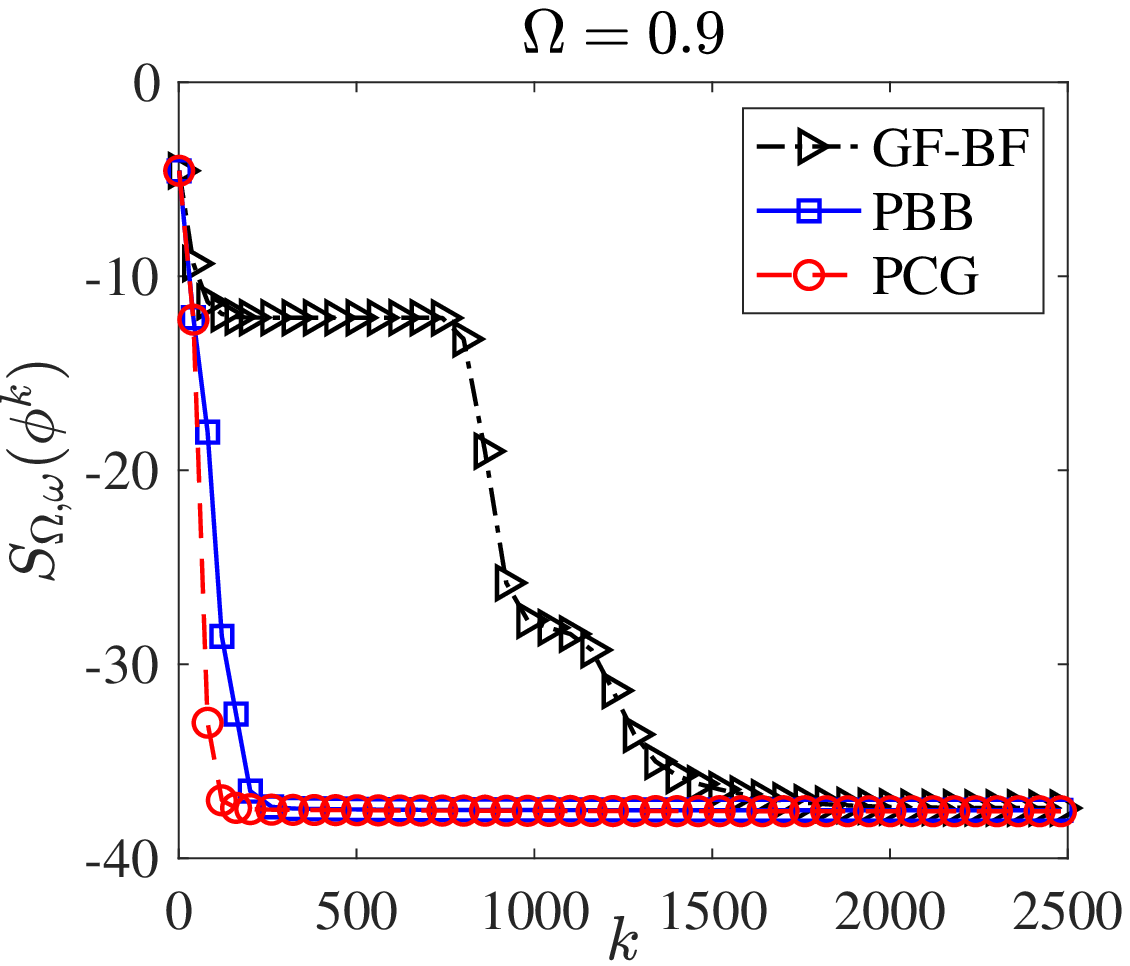,width=2.8in,height=2.in,angle=0}
}
\caption{The change of $S_{\Omega,\omega}(\phi^k)$ with respect to the number of iterations $k$ under  $\Omega=0.5$ (left) and $\Omega=0.9$ (right) in Example \ref{example-4}.}\label{fig_actions}
\end{figure}
\begin{table}[t!]
\centering\small
\caption{Comparison of GF-BF, PBB, PCG in Example \ref{example-4}.}
\label{tab:defoC_a}
\begin{tabular}{ccccccc}
\hline
Method & $\Omega$ & $iter$ & CPU(s) & $S_{\Omega,\omega}(\phi_g)$ & $\mathcal{E}_{err}^g$ & $r_{err}^{g,\infty}$\\
\hline
\multirow{5}{*}{GF-BF}
& 0.2 & 283 & 3.40 & -10.0938 & 8.40E-13& 1.03E-06   \\
& 0.3 & 287 & 3.22 & -10.1185 & 9.69E-13& 1.03E-06    \\
& 0.5 & 2204 & 24.08 & -11.2055& 9.30E-13&2.85E-06  \\
%& 0.6 & 94240 & 800.55 & -12.5869& 9.96E-13&1.77E-06   \\
& 0.7 & 197813& 1687.37& -15.1408 & 9.84E-13& 1.39E-06   \\
& 0.8 & 805823 & 6889.51 &  -20.6026 & 9.30E-13& 1.20E-06   \\
& 0.9 & 4120945 & 35367.80 & -37.5733 & 9.80E-13& 1.92E-06  \\
\hline
\multirow{5}{*}{PBB}
& 0.2 & 56 & 9.58 & -10.0938 & 2.29E-13& 8.03E-07   \\
& 0.3 & 53 & 6.30 & -10.1185 & 5.11E-13& 1.56E-06   \\
& 0.5 & 194 & 13.73& -11.2055 & 5.68E-13& 1.83E-06  \\
%& 0.6 & 2714 & 170.51 &-12.5869 & 9.89E-13& 1.50E-06   \\
& 0.7 & 6030 & 365.04 &-15.1408 & 8.70E-13& 1.59E-06   \\
& 0.8 & 11185 & 668.22 & -20.6026 &9.91E-13& 9.70E-07   \\
& 0.9 & 36376 & 2165.31 & -37.5733 & 5.04E-13& 7.23E-06  \\
\hline
\multirow{5}{*}{PCG}
& 0.2 & 35 & 7.07 & -10.0938 &8.68E-13& 2.06E-06   \\
& 0.3 & 35 & 9.94 & -10.1185 & 1.88E-13& 7.38E-07  \\
& 0.5 & 159 & 18.68 & -11.2055 & 7.21E-13& 1.07E-06  \\
%& 0.6 & 860 & 102.96 & -12.5869 & 9.50E-13& 3.96E-06  \\
& 0.7 & 1191& 141.83 & -15.1408 & 9.98E-13&1.32E-06   \\
& 0.8 & 2708 & 438.10 & -20.6026 & 8.56E-13& 2.59E-06  \\
& 0.9 & 6917 & 828.25 & -37.5733 & 8.52E-13& 8.78E-07  \\
\hline
\end{tabular}
\end{table}
\begin{figure}[!htp]
\centerline{
\psfig{figure=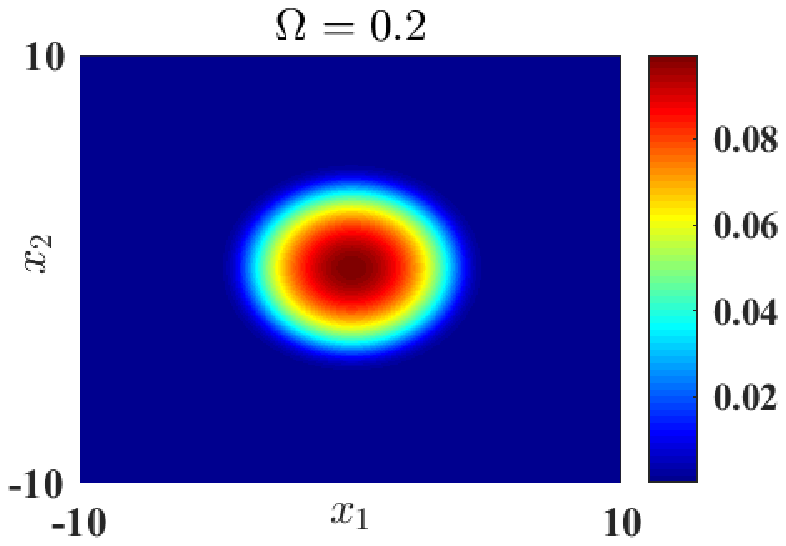,width=1.9in,height=1.7in,angle=0}
\hspace{-0.1in}
\psfig{figure=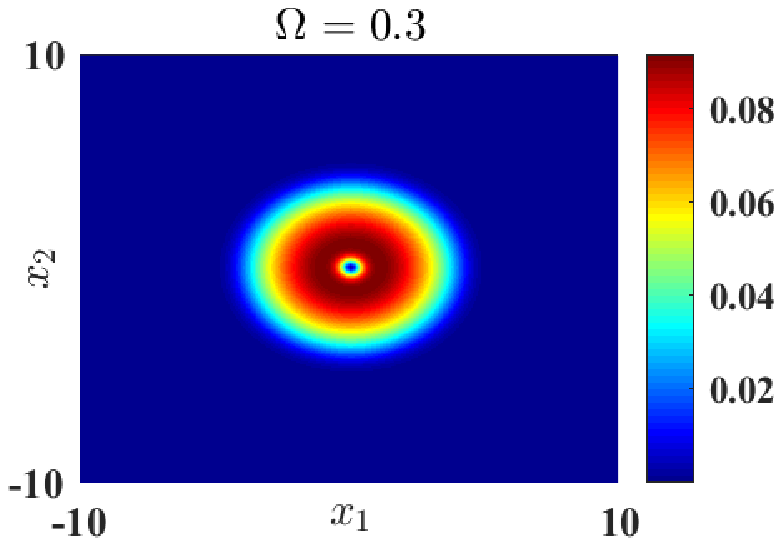,width=1.9in,height=1.7in,angle=0}
\hspace{-0.1in}
\psfig{figure=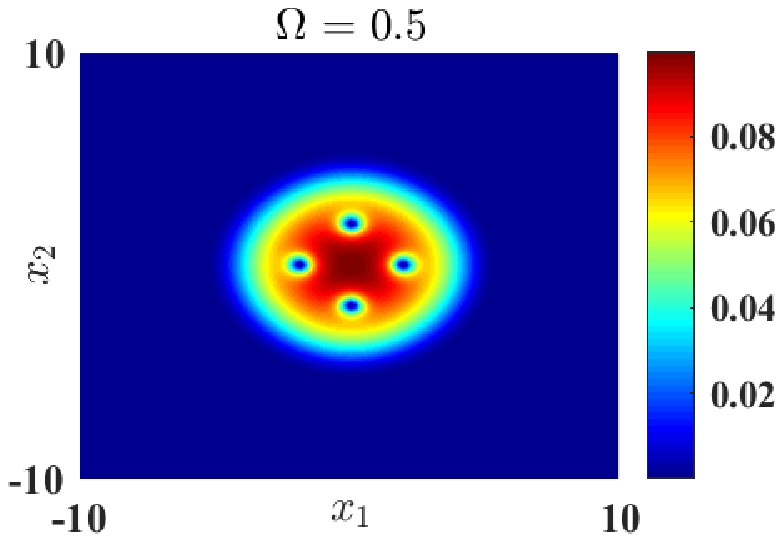,width=1.9in,height=1.7in,angle=0}
}
\centerline{
\psfig{figure=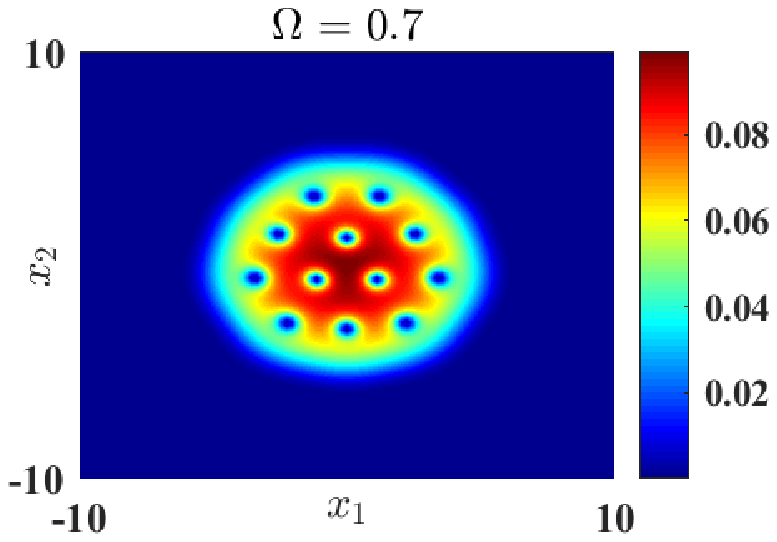,width=1.9in,height=1.7in,angle=0}
\hspace{-0.1in}
\psfig{figure=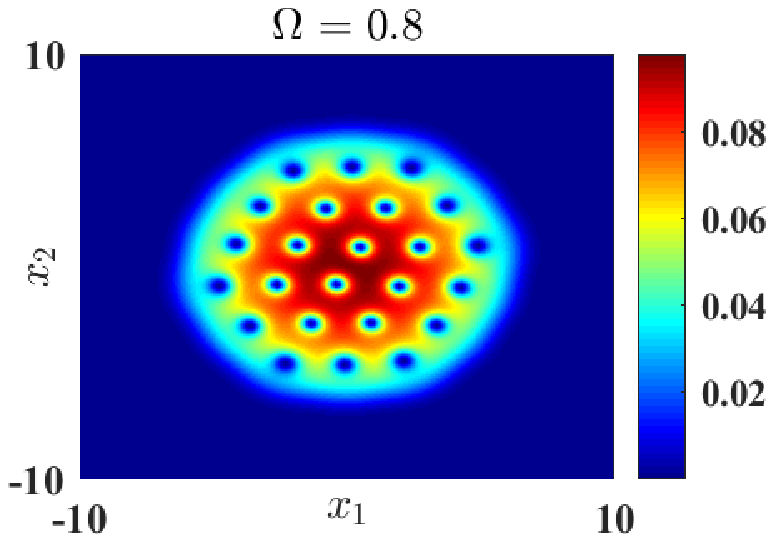,width=1.9in,height=1.7in,angle=0}
\hspace{-0.1in}
\psfig{figure=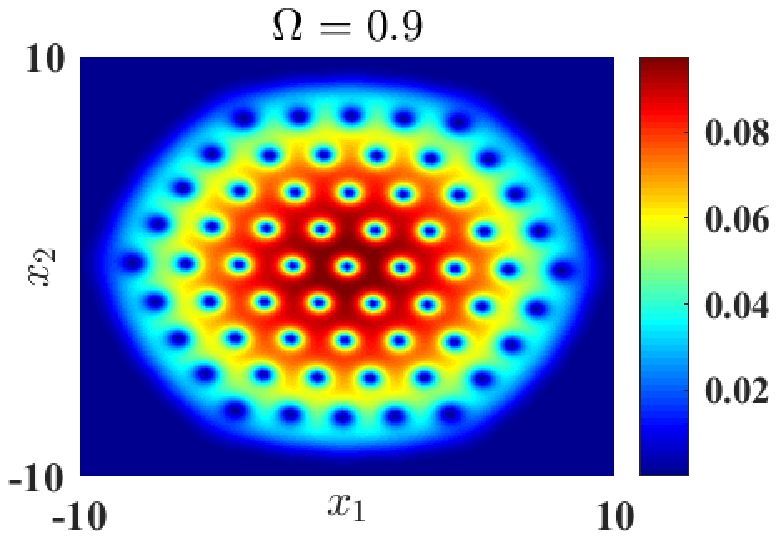,width=1.9in,height=1.7in,angle=0}
\vspace{-0.1in}
}
 \caption{Contour plots of $|\phi_g|^2$ for different $\Omega$ in Example \ref{example-4}.}\label{fig_gs_defocusing}
\end{figure}

The computation domain is fixed as $U=(-10,10)^2$ with mesh size $h=1/8$. We solve the problem \eqref{defocuse thm} by the GF-BF method in Algorithm~\ref{algorithm-gfbf} with $\tau=0.1$ and by the PBB in Algorithm \ref{algorithm3} and PCG in Algorithm \ref{algorithm4} with the preconditioner \eqref{P-V-ab-defocu}. The stopping criterion is set as $\mathcal{E}_{err}^n \leq 10^{-12}$ (\ref{residual defoc}) for all methods.
To discuss the effect of the initial data on the final state of the algorithms,
Tab. \ref{tab:initial_eng} lists the value of the action obtained by the GF-BF, PBB and PCG with each one of \eqref{initial-T1}-\eqref{initial-T3}. The smallest value obtained among the set of functions is underlined with blue color which indicates the approximate ground state. From the result, we can see that by using the initial data from \eqref{initial-T1}-\eqref{initial-T3}, GF-BF, PBB and PCG can all get to the ground state for a wide range of $\Omega$. Different methods may need different choices, and the `right' choice of initial data depends on the rotational speed $\Omega$. Inappropriate choices might lead to other steady states. We remark that the underlined initial data in blue converges most quickly in the algorithms, and so the results below are obtained with them.

To compare the efficiency of the methods, Fig. \ref{fig_actions} shows the decrease of $S_{\Omega,\omega}(\phi^k)$ with respect to the number of iterations $k$. Tab. \ref{tab:defoC_a} presents the total number of iterations, the CPUs and (\ref{residual defoc}) of the methods for different $\Omega$. It can be seen that GF-BF, PBB and PCG are all working for computing the ground states in the defocusing case, and PCG and PBB are remarkably more efficient than GF-BF  especially for large $\Omega$ ($0.5<\Omega<1$). Among them, PCG takes the least number of iterations and least CPUs to converge.

Fig. \ref{fig_gs_defocusing} shows the profiles of the obtained ground states for $\Omega=0.2, 0.3, 0.5, 0.7, 0.8,0.9$. We find that the action ground states in the defocusing case possess quantized vortices for suitably large $\Omega$, and the number of them increases when $\Omega$ increases. There  exists a critical value $\Omega_c\in(0.2,0.3)$ in the example for the rotational speed, which determines the appearance of the vortex in the solution. In addition, Tab. \ref{tab:defoC_a} shows that the action value at the ground state deceases significantly when $\Omega$ gets large.

\subsection{Relation with energy ground state}
At last but not least, we apply the proposed methods to investigate numerically the relation between  the action ground state and the energy ground state, i.e., the diagram \eqref{diagram}.

\begin{figure}[!htp]
\centerline{focusing case}
\psfig{figure=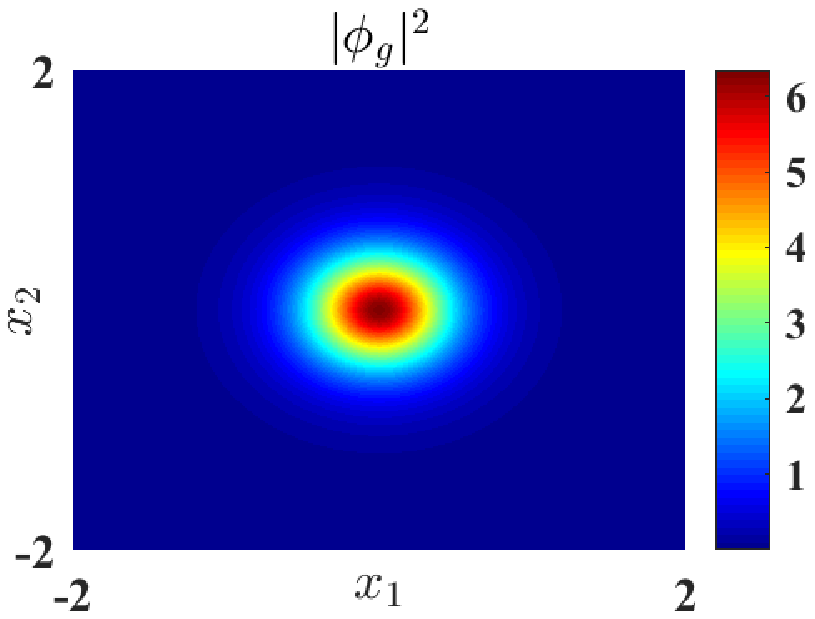,width=2.3in,height=1.7in,angle=0}
\psfig{figure=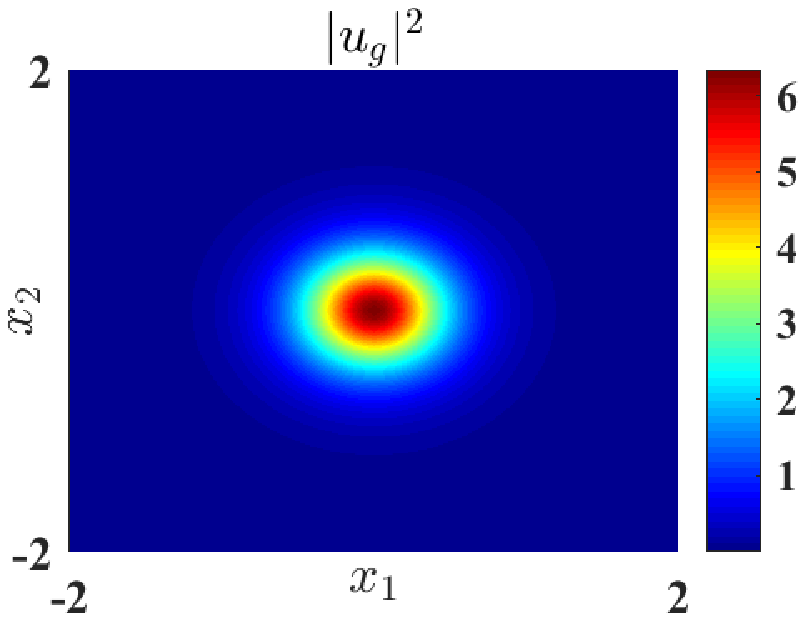,width=2.3in,height=1.7in,angle=0}\\
\centerline{defocusing case}
\psfig{figure=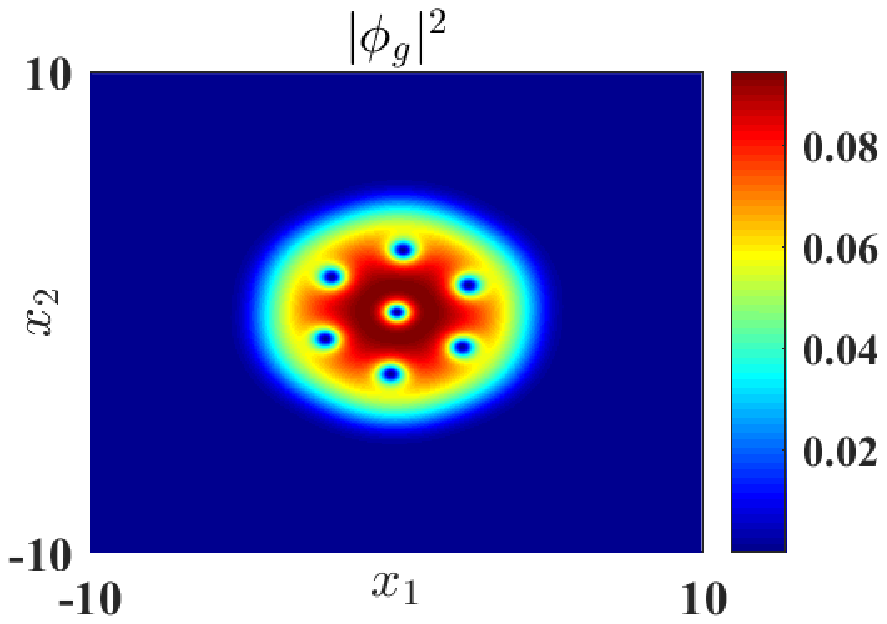,width=2.3in,height=1.7in,angle=0}
\psfig{figure=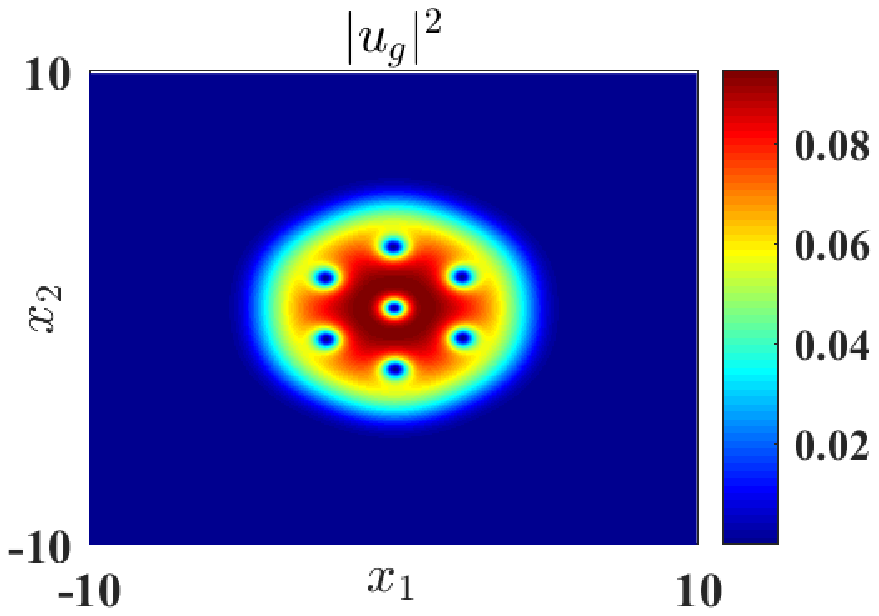,width=2.3in,height=1.7in,angle=0}
 \caption{Contour plots of $|\phi_g|^2$ from Prob. (I) (left column) and $|u_g|^2$ from Prob. (II) (right column) in focusing case and defocusing case.}\label{fig_relationship}
\end{figure}

We begin with the focusing case by using Example \ref{example-2} with $\Omega=0.6$. More precisely,  we first solve
$$
\mbox{Prob. (I)} \quad \phi_g(\bx)=\arg \min_{\phi\in\mathcal{M}} S_{\Omega,\omega}(\phi)
$$
by the PCG in Algorithm \ref{algorithm2} with $u_0= \phi_e/\|\phi_e\|_{L^{p+1}}$ defined in \eqref{initial-T3}. With the obtained action ground state $\phi_g$, we compute the mass $m=\|\phi_g\|^2_{L^2}$. Then, we solve the mass-prescribed  minimization of the energy (\ref{energy})
$$
\mbox{Prob. (II)} \quad u_g(\bx)=\arg \min_{\|u\|^2_{L^2}=m}  E_{\Omega}(u)
$$
by the PCG from \cite{ALT2017} and obtain the corresponding chemical potential $\omega(u_g)$  by \eqref{omega_g}. %We compare $\phi_g(\bf x)$ and $u_g(\bx)$ as well as the relationship between $\omega=1$ and $\omega(u_g)$ are studied.
Fig. \ref{fig_relationship} presents the solution $|\phi_g|^2$ and $|u_g|^2$ for Prob. (I) and Prob. (II), and our computation gives the  result:
$$
S_{\Omega,\omega}(\phi_g)=7.04363107 \approx E_{\Omega}(u_g)+\omega m=7.04363097, \quad \omega(u_g)=0.99999998\approx \omega=1.
$$
%Thus, we conclude numerically that, in the focusing case, the two kinds of minimization problems proposed in \eqref{diagram} are equivalent to each other.

Then we consider the defocusing case by Example \ref{example-4} with $\Omega=0.6$. We solve the Prob. (I) and Prob. (II)  similarly as above. The profiles of $|\phi_g|^2$ and $|u_g|^2$ are also given in Fig. \ref{fig_relationship}, and we find that
$$
S_{\Omega,\omega}(\phi_g)=-12.58691590 \approx E_{\Omega}(u_g)+\omega m=-12.58691588, \quad \omega(u_g)=-9.99999999\approx \omega=-10.
$$

It can be seen in both numerical experiments that, the ground state functions and the physical quantities of Prob. (I) and Prob. (II) are very close. This provides clues that the diagram (\ref{diagram}) may commutes under certain conditions. Further numerical and theoretical investigations are ongoing.

%\section{Conclusion}\label{sec:5}

\section*{Acknowledgements}
W. Liu is supported by NSFC 12101252, the International Postdoctoral Exchange Fellowship Program PC2021024 and the Guangdong Basic and Applied Basic Research Foundation 2022A1515010351. Y. Yuan is supported by NSFC 11971007 and 11601148. X. Zhao is supported by NSFC 12271413, 11901440 and the Natural Science Foundation of Hubei Province 2019CFA007.

\bibliographystyle{model1-num-names}

\end{document}